\documentclass[12pt,reqno]{amsproc}
\usepackage[T1]{fontenc}

\usepackage{amssymb}
\usepackage{amsmath}
\usepackage{amsthm}
\usepackage{amsfonts}
\usepackage{fullpage}
\usepackage{enumerate}
\usepackage{tikz-cd} 
\usepackage{comment}
\usepackage{hyperref}
\usepackage{float}
\usepackage{verbatim}
\usepackage{enumitem}
\usepackage{fancyhdr}
\usepackage{todonotes}
\usepackage[foot]{amsaddr}
\usepackage{svg}
\usepackage{multirow}
\usepackage{url}
\usepackage{textcomp}

\hypersetup{colorlinks}
\hypersetup{colorlinks=true, linkcolor=red, citecolor=blue, filecolor=blue, urlcolor=blue}

\newtheorem{thm}{Theorem}[section]
\newtheorem{cor}[thm]{Corollary}
\newtheorem{lem}[thm]{Lemma}
\newtheorem{prop}[thm]{Proposition}
\theoremstyle{definition}
\newtheorem{defn}[thm]{Definition}
\numberwithin{equation}{section} 
\numberwithin{table}{section}
\newtheorem{innercustomthm}{Theorem}
\newenvironment{customthm}[1]{\renewcommand\theinnercustomthm{#1}\innercustomthm}{\endinnercustomthm}

\newcommand{\nc}{\newcommand}
\newcommand{\rc}{\renewcommand}

\rc{\t}{\text}
\nc{\tb}{\textbf}
\nc{\mb}{\mathbb}
\nc{\tc}{\textcolor}
\nc{\mf}{\mathfrak}
\nc{\mc}{\mathcal}
\nc\on{\operatorname}

\nc{\CC}{\mathbb{C}}
\nc{\EE}{\mathbb{E}}
\nc{\FF}{\mathbb{F}}
\nc{\II}{\mathbb{I}}
\nc{\NN}{\mathbb{N}}
\nc{\QQ}{\mathbb{Q}}
\nc{\RR}{\mathbb{R}}
\rc{\SS}{\mathbb{S}}
\nc{\ZZ}{\mathbb{Z}}

\rc{\d}{\on{d}}
\nc{\Law}{\on{Law}}
\nc{\Prob}{\mathbb{P}}
\nc{\Unif}{\on{Unif}_d}
\nc{\Wass}{\on{W}}
\nc{\Proj}{\on{P}}
\nc{\Proportion}{\varrho}
\nc{\nutwomass}{\mathfrak{m}}
\nc{\specgap}{\lambda^{\on{gap}}}
\nc{\pran}{{\underline{\measuredangle}}}

\nc{\OurMeasures}{\mc{P}}
\nc{\Ball}{\mc{B}}
\nc{\ExpBall}{\mathring{\mc{B}}}
\nc{\Complement}{{\on{c}}}
\nc{\Set}[1]{\mathbf{#1}}
\nc{\MeasRes}[2]{{#1}|_{#2}}

\nc{\explainbox}[2]{
\vspace{3mm}
\noindent
\fbox{
    \parbox{\textwidth-6mm}
        {
            \vspace{1mm}
            \textbf{\underline{#1}} {#2} 
            \vspace{1mm}
        }
    }
\vspace{3mm}
}

\nc{\QtwoDefinition}{\bigg( \frac72 + 8\sqrt{2}d \bigg) + \frac{(27/2)d\varphi_{\max} + 6\alpha \tau }{\Phi}}

\begin{document}

\title{Tangent Space and Dimension Estimation \\ with the Wasserstein Distance}
\author[U. Lim]{Uzu Lim}
\email{lims@maths.ox.ac.uk}
\author[H. Oberhauser]{Harald Oberhauser}
\email{oberhauser@maths.ox.ac.uk}
\author[V. Nanda]{Vidit Nanda}
\address{Mathematical Institute, 
University of Oxford, Radcliffe Observatory, Andrew Wiles Building, Woodstock Rd, Oxford OX2 6GG}
\email{nanda@maths.ox.ac.uk}

\begin{abstract}
	Consider a set of points sampled independently near a smooth compact submanifold of Euclidean space. We provide mathematically rigorous bounds on the number of sample points required to estimate both the dimension and the tangent spaces of that manifold with high confidence. The algorithm for this estimation is Local PCA, a local version of principal component analysis. Our results accommodate for noisy non-uniform data distribution with the noise that may vary across the manifold, and allow simultaneous estimation at multiple points. Crucially, all of the constants appearing in our bound are explicitly described. The proof uses a matrix concentration inequality to estimate covariance matrices and a Wasserstein distance bound for quantifying nonlinearity of the underlying manifold and non-uniformity of the probability measure.
\end{abstract}
\maketitle

\section{Introduction}

In this paper, we study the problem of estimating tangent spaces and the intrinsic dimension of a data manifold with high confidence. Our goal is to provide mathematically rigorous, explicit and practical bounds on the number of sample points required for such estimations. In data science terms, a tangent space gives the optimal local linear regression and the intrinsic dimension is the degree of freedom of data. Our estimators are standard applications of Local PCA, a local version of \textit{principal component analysis} (PCA). Locally computed principal components approximate tangent spaces, and their eigenvalues allow inference of the intrinsic dimension. 

To the best our knowledge, our results on \textit{both} tangent space and dimension estimation are the first ones which simultaneously: (1) apply to noisy non-uniform distribution concentrated near a manifold, with the noise term allowed to vary across the manifold, (2) accommodate multiple data points, and (3) explicitly compute all constants appearing in the bounds, including dependence on dimension. Our proofs clearly separate the geometric and probabilistic aspects of the estimation process into modular components; we hope that the reader will find this convenient when attempting to use, build upon or improve our results. We begin by defining our estimators.

\explainbox{Estimators from Local PCA.}{
Given $m$ points $\mathbf x = \{x_1, \ldots x_m\} \subset \RR^D$, denote by $\bar x = \frac1m \sum_i x_i$ the mean and denote by $\hat\Sigma[\mathbf x ] = \frac1{m} \sum_i (x_i - \bar x)(x_i - \bar x)^\top$ the empirical covariance matrix. By PCA we mean the diagonalisation $\hat\Sigma[\mathbf x] = U \Lambda U^\top$, where $U$ is an orthogonal matrix and $\Lambda$ is a diagonal matrix. Writing $U = [v_1, \ldots v_D]$ and letting diagonal entries of $\Lambda$ be $\lambda_1 \ge \ldots \ge \lambda_D \ge 0$, we define lower-dimensional subspaces and eigenvalues as:
\begin{align*}
	\Pi_k[\mathbf x] &:= \on{span}(v_1, \ldots, v_k) \\
	\vec \lambda \hat\Sigma[\mathbf x] &:= (\lambda_1, \ldots, \lambda_D)
\end{align*}

Local PCA at an open set $W \subseteq \RR^D$ performs PCA on points of $\mathbf x$ that lie in $W$. We are interested in $W$ given by an open ball. Given a radius parameter $r>0$, let $\mathbf{x}_i := \{x_j \> | \> j \neq i \} \cap \{y \> | \> \|y-x_i\| < r \}$. Define the \textit{$k$-dimensional tangent space estimator} and the \textit{intrinsic dimension estimator} with threshold $\eta$:
\begin{align}
    \hat \Pi(\mathbf x, r, i, k) :=& \Pi_k[\mathbf{x}_i] \nonumber \\
	\hat d(\mathbf x, r, i, \eta) :=& \on{Thr} \big( \vec\lambda\hat\Sigma[\mathbf{x}_i], \> \eta \big) \label{estimator defn}
\end{align}
where $\on{Thr} \big( (\lambda_1, \ldots \lambda_D), \eta\big)$ is the smallest $k$ such that $(\lambda_{k+1} + \cdots + \lambda_D) \le \eta \cdot (\lambda_1 + \cdots + \lambda_D)$.
}

When we calculate $\hat\Pi$ and $\hat d$ for a sample drawn near a $d$-dimensional manifold, we will get accurate estimations of tangent spaces and the intrinsic dimension $d$. Intuitively, this is because when a manifold is zoomed in closely enough at each point, its curvature flattens out and we essentially get a $d$-dimensional disk. Let's translate this intuition to precise mathematics. To do this, we precisely describe how we draw a random sample near a manifold.

\begin{figure}
    \centering
    \includegraphics[scale=0.4]{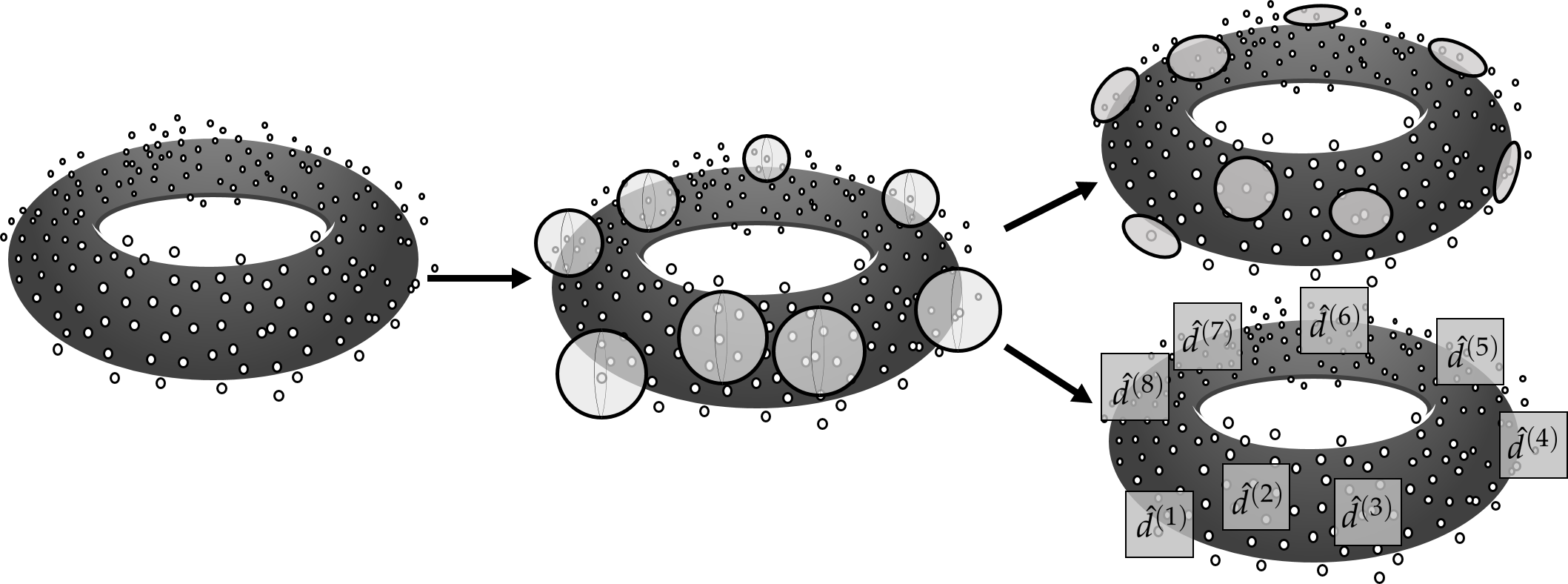}
    \caption{An illustration of Local PCA. Left: Dataset concentrated near a torus. Middle: Local neighborhood selection. Top Right: Tangent space estimation. Top bottom: Dimension estimation.}
    \label{pipeline}
\end{figure}

\explainbox{Setup.}{
Let $M \subset \RR^D$ be a smoothly embedded $d$-dimensional compact manifold. Let $\mu_0$ be a Borel probability measure on $\RR^D$ with a probability density function $\varphi: M \rightarrow \RR_{\ge 0}$: for each open $U \subseteq \RR^D$, define
\[\mu_0(U) := \int_{U \cap M} \varphi ~ \d \mc H^d\]
where $\mc H^d$ is the $d$-dimensional Hausdorff measure. Let $X \sim \mu_0$. Let $Y$ be a $\RR^D$-valued random variable representing noise, with bounded norm $\|Y\| \le s$. Now our random sample $\textbf{X} = \{X_1, \ldots X_m\}$ is drawn i.i.d. from $\mu$:
\[\mu := \Law(X+Y)\]
Here we emphasise that $X$ and $Y$ are \textit{not assumed to be independent}. Assume that $\varphi$ satisfies the Lipschitz condition $\|\varphi(x) - \varphi(y)\| \le \alpha \cdot \d_M(x,y)$ for every $x, y \in M$, where $\d_M$ is the geodesic distance on $M$. Assume that $s < \tau$, where $\tau$ is the reach of $M$, defined as the maximum length to which $M$ can be thickened normally without self-intersection.
}

Additionally, denote by $\omega_d = \pi^{d/2}/\Gamma(\frac d2+1)$ the volume of the unit $d$-dimensional ball. Denote by $\measuredangle(\Pi_1, \Pi_2)$ the principal angle between subspaces $\Pi_1, \Pi_2$ (Definition \ref{principal angle}). Denote by $\Prob(E)$ the probability of event $E$. Denote by $\varphi_{\max}, \varphi_{\min}$ the maximum and the minimum of the function $\varphi$. Our main results ensure accurate estimations if:
\begin{enumerate}
    \item $r$ is small enough to ignore curvature
    \item $r$ is big enough to ignore noise
    \item $m r^d$ is big enough to ensure dense sampling
\end{enumerate}

\explainbox{Main Results.}{
\begin{customthm}{A}[Tangent Space Estimation]\label{thmA}
    Let $\mathbf X = \{X_1, \ldots X_m\}$ be a random sample as above. Given $\theta, \delta, \Proportion > 0$, the following holds:
\begin{align*}
    & \sqrt{2\tau s} \le r \le S_1 \quad \text{and} \quad \frac {m r^d}{\log m} \ge S_2 \implies \Prob\bigg( \max_{i \le \Proportion m} \measuredangle\left( \widehat{T}_i , T_i \right) \le \theta \bigg) \ge 1-\delta
\end{align*}
Here $T_i$ is the tangent space of $M$ at $X_i^\perp$, the orthogonal projection of $X_i$ to $M$. $\widehat{T}_i = \hat \Pi(\mathbf X, r, i, d)$ is the tangent space estimator defined in \eqref{estimator defn}. $S_1, S_2$ are defined as:
\begin{align*}
    S_1(\tau, d, \varphi, \theta) \>=\>& \frac{\sin\theta}{(d+2)^{3/2}} \frac{\varphi_{\min}}{c_1 d\varphi_{\max} + c_2 \alpha\tau} \\
    S_2(\varrho, D, d, \varphi, \theta) \>=\>& \frac{c_3(d+2)^3}{\omega_d \varphi_{\min} \sin^2\theta} \log \bigg( \frac{c_4 D\Proportion}{\delta} \bigg)
\end{align*}
where $(c_1, c_2, c_3, c_4) = (928, 192, 18574, 14)$.
\end{customthm}
}

\explainbox{}{
\begin{customthm}{B}[Intrinsic Dimension Estimation]\label{thmB}
    Let $\mathbf X = \{X_1, \ldots X_m\}$ be a random sample as above. Given $\eta, \delta, \Proportion > 0$ with $\eta < (2D)^{-1}$, the following holds:
\begin{align*}
    & \sqrt{2\tau s} \le r \le S_1 \quad \text{and} \quad \frac {mr^d}{\log m} \ge S_2 \implies \Prob\bigg( \hat d_i = d \text{ for } i \le \varrho m \bigg) \ge 1-\delta
\end{align*}
where $\hat d_i = \hat d(\mathbf X, r, i, \eta)$ is the dimesnion estimator defined in \eqref{estimator defn}. 
\begin{align*}
S_1(\tau, d, \varphi, \eta) \>=\>& \frac1{(d+2)D(1+\eta^{-1})}\frac{\varphi_{\min}}{c_1 d\varphi_{\max} + c_2 \alpha \tau} \\
S_2(\varrho, D, d, \varphi, \eta) \>=\>& \frac{c_3 (d+2)^2 D^2 (1+\eta^{-1})^2}{\omega_d \varphi_{\min}} \log\bigg( \frac{c_4 D\Proportion}{\delta} \bigg)
\end{align*}
where $(c_1, c_2, c_3, c_4) = (1392, 288, 41791, 14)$.
\end{customthm}
}

\vspace{3mm}
\textbf{Remarks.} If $\varphi$ vanishes in a small region, we may avoid division by zero by replacing $\varphi_{\min}$ by $\Phi(r_-)$. Here $\Phi$ quantifies local concentration of the measure $\mu_0$. It is defined as $\Phi(r) = \inf_{x \in M} \mu_0 \big ( U_{x,r} \big) / (\omega_d r^d)$ and $U_{x,r} = \{ y \in M \>|\> \d(x, \Pi_x(y)) \le r \}$, where $\Pi_x$ is the projection map to $T_x M$. Also $r_-$ is defined as $r_- = r(1-r^2/4\tau^2) - 2s$. This stronger result is stated in Theorem \ref{main}. Also, conditions for $r$ given by two inequalities can be collectively replaced by one upper bound on a function $Q$, defined in Proposition \ref{main transport}. Lastly, a special case of our result is given by setting $r = (S_2 \log m / m)^{1/d}$, which makes our results directly comparable to Theorem 2 of \cite{aamari_levrard}. The constant $S_2$ is fully calculated in our main theorems, improving Theorem 2 of \cite{aamari_levrard}.

\subsection{Structure of the paper}

Theorems A and B follow easily from Theorem \ref{main} in Section 5, which is about estimating covariance matrices locally. Ingredients for its proof span Sections 2, 3, 4. In Section 2, we modify the matrix Hoeffding's inequality to show that Local PCA correctly estimates covariance (Proposition \ref{main empirical covariance estimation}). In Section 3, we show that given two compactly supported probability measures $\mu, \nu$ valued in $\RR^D$, there is a Lipschitz relation of the form $\|\Sigma[\mu] - \Sigma[\nu]\| \le C \cdot \Wass_1(\mu, \nu)$ where $\Sigma[\mu]$ is the covariance matrix of $\mu$ (Proposition \ref{covariance lipschitz}). In Section 4, we show that if a well-behaved measure on a manifold is restricted to a tiny ball, then its Wasserstein distance to the uniform measure over the unit tangential disk is small (Proposition \ref{main transport}). The Lipschitz relation in Section 3 then translates the Wasserstein bound to the bound on matrix norms. 

We summarize the notations and conventions of this article in the Appendix (page \pageref{notation appendix}).

\begin{figure}[h!]
    \centering
    \includegraphics[scale=0.4]{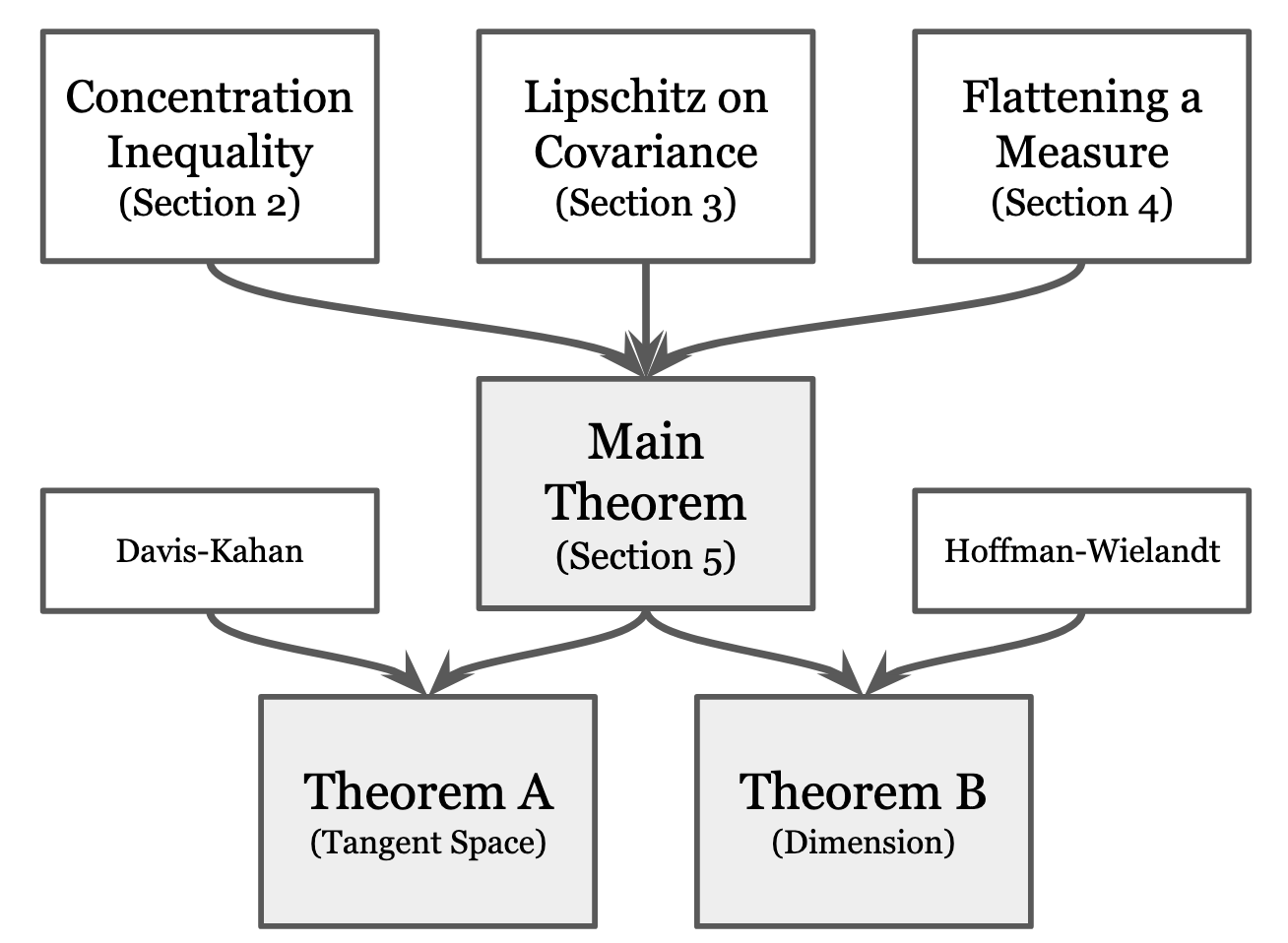}
    \caption{Summary of the relations between the main results.}
    \label{main schematic 2}
\end{figure}

\subsection{Related works}

The task of estimating geometric and topological quantities of manifolds from finitely many sample points lies at the crux of statistical inference, and as such the literature surrounding these topics is vast. Below we have described some of the techniques of which we are aware, and direct the reader to \cite{wasserman_tda, lee_nldr, chazal_intro} for a more comprehensive survey.

\textbf{Tangent space estimation.} Probabilistic bounds on tangent space estimation using Local PCA have been studied in considerable detail, for example in \cite{aamari_levrard, tyagi_vural_frossard, kaslovsky_meyer, singer_wu}. To the best of our knowledge, our work is the first in which the tangent space estimation applies to:
\begin{enumerate}
    \item Noisy non-uniform distribution with noise allowed to vary across the manifold,
    \item Deals with multiple data points simultaneously, and
    \item Explicitly computes all constants in bounds, including dimensional dependence.
\end{enumerate}
The dimensional dependence, for example, reflects the fact that covariance of the uniform distribution over the $d$-dimensional unit disk have $O(1/d)$ terms (see Lemma \ref{covariance of disk}). 

In \cite{kaslovsky_meyer} and \cite{tyagi_vural_frossard}, the underlying probability measure is assumed to be uniform, and only estimation at a single point is considered. In \cite{singer_wu}, various constants have not been explicitly computed, and there is no consideration of noise in data distribution. In \cite{aamari_levrard}, various constants have not been computed explicitly, thus not specifying the minimum sample size requirement and scaling factor $c$ for their prescription $r = (c \log m/m)^{1/d}$. Furthermore, their noise model is assumed to be orthogonal to the manifold. 

\vspace{3mm}
\textbf{Dimension estimation.} 
The idea to use local principal component analysis for estimating intrinsic dimension is ancient, dating back at least to \cite{fukunaga_olsen}. As such, there is a plethora of literature on the problem of estimating intrinsic dimensions. The work of \cite{levina_bickel} provides a practical and widely-used maximum likelihood estimator, but there are no known theoretical guarantees of its correctness even for synthetic data. The minimax-based estimator of \cite{kim_minimax} does come with such guarantees, but in order to compute it one is compelled to solve minimisation problems over the symmetric group on $m$ elements (with $m$ being the total size of the input dataset); thus, this estimator becomes intractable in practice. The recent work of \cite{rakhlin_gan} introduces a far more efficient Wasserstein-based estimator with guarantees\footnote{We note in passing that the number of points we require to ensure a $1-\delta$ probability of correct dimension estimation in our result is $m \sim \log(1/\delta)$, which improves on the rate $m \sim \log(1/\delta)^3$ of \cite{rakhlin_gan}.}, but does not adapt to noise. Our efforts in this paper were motivated by the desire to find a suitable balance between practical efficiency, theoretical soundness and compatibility with noise.

\vspace{3mm}
\textbf{Concentration inequality.} Our concentration inequality for covariance matrices, Proposition \ref{conc of cov}, is directly derived from the matrix Hoeffding inequality in \cite{tropp_user_friendly}. A more sophisticated approach, such as the one from  \cite{koltchinskii_lounici}, may be used to improve our concentration inequality. For instance, the constants appearing in Proposition \ref{conc of cov} may be improved. Similar methods for analyzing (non-local, non-manifold) PCA are also studied in \cite{koltchinskii_lounici_2, reiss}.

\vspace{3mm}
\textbf{Other Techniques.} We also list related techniques that appear in other papers. A cubic bound of the form $\|\Sigma[\mu] - \Sigma[\nu]\| \le C r^3$, where $\mu, \nu$ are probability measures supported on a ball of radius $r$ in $\RR^D$, is derived for uniform measures in \cite{spectral_clustering_local_pca}. We also obtain a similar inequality (Proposition \ref{covariance lipschitz} and Corollary \ref{main transport simplified}). The key difference in the two derivations is that our approach uses the Wasserstein distance rather than the total variation distance from \cite{spectral_clustering_local_pca} to quantify similarity of measures. Our inequality  has the advantage of allowing non-uniformity and of having explicit constants. 

We use a transportation plan in Proposition \ref{main transport} to quantify how much a measure supported near a manifold locally deviates from the uniform measure on a tangential disk. This transportation plan is executed with a similar idea as the proof of Proposition 3.1 in \cite{tinarrage}. However, their transportation plan does not involve noise and applies to different types of local covariance matrices. 

In \cite{aizenbud-sober}, local polynomial regression were used to estimate manifolds and their tangent spaces from uniform point samples lying on tubular neighbourhoods. Compared to this work, our results have the advantage of not requiring the noise to be uniformly distributed. Our result only estimates tangent spaces and not higher-order information like curvature. However, the Wasserstein bound could potentially be leveraged to produce bounds on polynomial approximations.

Local PCA has been extensively used in contexts independent of the manifold hypothesis \cite{fukunaga_olsen, kambhatla_leen, chem_pca_analysis, media_lab_local_pca}, although the theoretical analysis is either heuristic or makes strong assumptions on the underlying distribution (e.g. Gaussian). Theoretical analysis in manifold learning is a flourishing field, with many significant examples including \cite{genovese1, genovese2, aamari2, aamari_levrard, testing_manifold, putative_manifold, kim_minimax, aizenbud-sober, tinarrage} and many others.

\vspace{5mm}
{\footnotesize
\subsection*{Acknowledgements} ~\newline
We are grateful to Eddie Aamari, Yariv Aizenbud, Barak Sober and Hemant Tyagi for valuable discussions.

\indent UL is supported by the Korea Foundation for Advanced Studies. 
\newline
\indent VN is supported by the EPSRC Grant EP/R018472/1. 
\newline
\indent HO is supported by the EPSRC grant “Datasig” [EP/S026347/1], The Alan Turing Institute, and the Oxford-Man Institute.
}

\section{Local estimation of covariance matrices}\label{chap2}

The main result of this section is Proposition \ref{main empirical covariance estimation}, where we establish bounds for local covariance estimation. Our main tool is the {\em matrix Hoeffding inequality} \cite[Theorem 1.3]{tropp_user_friendly}\footnote{Our version of the matrix Hoeffding inequality follows from the one in \cite{tropp_user_friendly} by noting that for any matrix $A$, the operator norm $\|A\|$ equals $\max(\lambda_{\max}(A), \lambda_{\max}(-A))$ where $\lambda_{\max}$ denotes the largest eigenvalue. And moreover, $\|A\| \le \alpha$ implies that $\alpha^2 \cdot \text{Id} - A^2$ is positive definite.}. Here onwards, we will use $\|A\|$ to denote the {operator norm} of a given matrix $A$: $\|A\| := \sup_{\|x\|=1} \|Ax\|$.

\begin{thm}[Matrix Hoeffding]
	Let $Y_1, \ldots Y_m$ be independent Hermitian random $D \times D$ matrices so that for each $i$ we have both $\EE Y_i = 0$ and $\|Y_i \| \le \alpha_i$ for some real number $\alpha_i \geq 0$. Write $\sigma^2 = \sum_{i=1}^m\alpha_k^2$. Then for every $\epsilon \ge 0$,
	\[\Prob\big( \left\| Y_1 + \cdots + Y_m \right\| \ge \epsilon \big) \le 2 D \cdot \exp \left( \frac{-\epsilon^2}{8\sigma^2} \right)
	\]
\end{thm} \label{matrix hoeffding}
This inequality can be used to establish concentration of vectors.\footnote{Apply Hermitian dilation, which takes a rectangular matrix $A$ and produces a Hermitian matrix $A_H = \left[\begin{smallmatrix} 0 & A^\top \\ A & 0 \end{smallmatrix}\right]$. Then $\|A_H\|^2 = \|A_H^2\| = \|A\|^2$ and the result applies.}
\begin{cor}\label{vector concentration}
    Let $X_1, \ldots X_m$ be independent random vectors in $\RR^D$ satisfying $\EE X_i = 0$, and $\|X_i\| \le \alpha_i$ for some real number $\alpha_i$. Write $\sigma^2 = \sum_1^m\alpha_i^2$. Then for every $\epsilon \ge 0$,
	\[\Prob\big( \left\| Y_1 + \cdots + Y_m \right\| \ge \epsilon \big) \le 2 (D+1) \cdot \exp \left( \frac{-\epsilon^2}{8\sigma^2} \right)
	\]
\end{cor}

Throughout the remainder of this section, we fix a Borel probability measure $\mu$  on $\RR^D$. We define some probabilistic notions.

\begin{defn}
    Given $X \sim \mu$, the \textit{covariance matrix} of $\mu$ is the following $D \times D$ matrix:
    \[
    \Sigma[\mu] := \EE[(X - \EE X)(X - \EE X)^\top]
    \]
    Let $\delta_x$ be the Dirac delta measure at a point $x$. Given $\mathbf x = \{x_1, \ldots x_m\} \subset \RR^D$, define the empirical measure $\delta_{\mathbf x}$:
    \[ \delta_{\mathbf x} := \frac1m(\delta_{x_1} + \cdots + \delta_{x_m}) \]
    Given a Borel set $U \subseteq \RR^D$, the {\em normalised restriction} of $\mu$ to $U$ is defined as follows: for each Borel set $V \subset \RR^D$,
    \begin{align*}
        \MeasRes \mu U (V) := \frac{\mu(U \cap V)}{\mu(U)}
    \end{align*}
    We impose the convention that $\MeasRes \mu U = 0$ whenever $\mu(U) = 0$, and note that $\MeasRes{\mu}{U}$ constitues a Borel probability measure on $\RR^D$ whenever $\mu(U) > 0$.
\end{defn}

If $\mathbf X = (X_1, \ldots X_m)$ is $\mu$-i.i.d. sample, then $\Sigma[\delta_{\mathbf X}] = \frac1m \sum_{i=1}^m (X_i - \bar X)(X_i - \bar X)^\top$, where $\bar X = \frac1m \sum_i X_i$ is the sample mean. The expected value of $\Sigma[\delta_{\mathbf X}]$ is in fact $\frac{m-1}m \Sigma[\mu]$, but the following computation tells us that we may use it to estimate $\Sigma[\mu]$.

\begin{prop}[Concentration inequalities for covariance]\label{conc of cov}
	Let $\mu$ be a Borel probability measure on $\RR^D$ and let $\tb X = (X_1, \ldots X_m)$ be an i.i.d. sample drawn from $\mu$. Suppose that the support of $\mu$ is contained in a ball of radius $r$. Then for each $\epsilon \ge 0$,
	\begin{align*}
	    \Prob \big(\| \hat\Sigma_0 - \Sigma[\mu] \| \ge \epsilon\big) &\le 2D \cdot \exp \left(-\frac{m\epsilon^2}{512r^4} \right) \\ 
		\Prob\big(\| \hat\Sigma - \Sigma[\mu]  \| \ge \epsilon\big) &\le (4D+2) \cdot \exp\left( - \frac{m\epsilon^2}{1152r^4} \right)
	\end{align*}
	where, denoting $\bar X = \frac1m \sum_i X_i$,
	\begin{align*}
	    \hat\Sigma_0 &= \frac1m \sum_{i=1}^m (X_i - \EE X)(X_i - \EE X)^\top, \quad \hat\Sigma = \frac1m \sum_{i=1}^m (X_i - \bar X)(X_i - \bar X)^\top
	\end{align*}
\end{prop}
\begin{proof}
We may assume that $r=1$ without loss of generality, since for general $r$ we know that $r^2 \Sigma$ is the covariance of $r \cdot X$ for all $X \sim \mu$. Thus, we have $\|X - \EE X\| \leq 2$ by the triangle inequality and the constraint on the support of $\mu$. The bound for $\hat\Sigma_0$ is obtained directly by applying the matrix Hoeffding inequality from Theorem \ref{matrix hoeffding} as follows. Writing $\Sigma[\mu] = \Sigma$, set $Y_i = \frac1m ((X_i-\EE X) (X_i-\EE X)^\top - \Sigma)$. Then $\|Y_i\| \le (4+4)/m$ and $\sigma^2 = m \cdot (8/m)^2 = 64/m$. Since $\hat\Sigma_0 = \hat\Sigma + (\bar X - \EE X)(\bar X - \EE X)^\top$, we have
\begin{align*}
    \Prob(\| \hat \Sigma - \Sigma \| \ge t )
    = \Prob(\| \hat \Sigma_0 - (\bar X - \EE X)(\bar X - \EE X)^\top - \Sigma \| \ge t ). 
\end{align*}
Therefore, for any parameter $\alpha$ in $[0,1]$, we obtain
\begin{align*}
    \Prob\big(\| \hat \Sigma - \Sigma \| \ge t \big) &\le  \Prob\big(\| \hat \Sigma_0 - \Sigma \| \ge \alpha t \big) + \Prob\big( \| \bar X - \EE X \|^2 \ge (1-\alpha) t \big) \\
    & \le \Prob\big(\| \hat \Sigma_0 - \Sigma \| \ge \alpha t \big) + \Prob\left( \| \bar X - \EE X \| \ge \frac12(1-\alpha)t \right) \\
    & \le 2D \cdot \exp\left(-\frac{\alpha^2 mt^2}{512}\right) + 2(D+1) \cdot \exp\left(-\frac{(1-\alpha)^2 m t^2}{128}\right).
\end{align*}
In the last inequality, we used the bound for $\hat\Sigma_0$ as well as Corollary \ref{vector concentration}, with $\sigma^2 = 4$. Choosing $\alpha = 2/3$ to make the exponents equal, we obtain the second bound.
\end{proof}

We will estimate $\Sigma[\MeasRes{\mu}{U}]$ with $\Sigma[\delta_{\Set X}|_U]$ assuming that $U$ is bounded.

\begin{prop}\label{fixed U}
	Let $\Set X = (X_1, \ldots X_m)$ be an i.i.d. sample drawn from $\mu$ and let $U \subseteq \RR^D$ be a Borel set which is contained in a ball of radius $r$. Denote by $\hat \Sigma_U$ the covariance $\Sigma[\delta_{\Set X}|_U]$, and similarly write $\Sigma_U = \Sigma[\MeasRes{\mu}{U}]$. Then for any error level $\epsilon > 0$, we have that $\hat \Sigma_U$ estimates $\Sigma_U$:
	\[
	\Prob\big(\| \hat \Sigma_{U}  - \Sigma_U \|  \le \epsilon\big) \ge 1-\delta,
    \]
	where $\delta$ is an expression such that $\lim_{m \rightarrow \infty} \delta = 0$, defined as:
	\[
	\delta = (4D+2) ( 1 - \mu(U) (1-\xi) )^m \quad \text{ with } \quad \xi := \exp(-\epsilon^2/1152r^4).
	\]
\end{prop}
\begin{proof}
	The proof follows from conditioning the membership of elements of $\Set X$ to $U$. Denoting by $\mc{S}_I$ the event $(X_i \in U \iff i \in I)$ and writing $u := \mu (U)$, we have
	\begin{align*}
		\Prob\big( \| \hat \Sigma_{U }  - \Sigma_U \|  \ge \epsilon\big) 
		&= \sum_{I \subseteq \{1, \ldots m\} } \Prob\big( \| \hat \Sigma_{U}  - \Sigma_U \|  \ge \epsilon \> | \> \mc S_I \big) \cdot \Prob(\mc S_I).
	\end{align*}
	Writing $|I|$ for the cardinality of each $I$, we have
	\begin{align*}
	    \Prob\big( \| \hat \Sigma_{U }  - \Sigma_U \|  \ge \epsilon\big)	&= \sum_{I \subseteq \{1, \ldots m\} } u^{|I|} (1-u)^{m - |I|} \Prob\big( \| \hat \Sigma_{U}  - \Sigma_U \|  \ge \epsilon \>|\> \mc{S}_I \big) \\
		&= \sum_{k=0}^m \binom m{k} u^k (1-u)^{m-k}  \Prob\big( \| \hat \Sigma_{U}  - \Sigma_U \|  \ge \epsilon \>|\> \mc{S}_{\{1, \ldots k\} }\big)   \\
		&\le  \sum_{k=0}^m \binom m{k} u^k (1-u)^{m-k} \cdot (4D+2) \xi^k \\
		&= (4D+2) \cdot ( 1 - u (1-\xi) )^m.
	\end{align*}
	Here Proposition \ref{conc of cov} was applied in the only inequality above. Note that the possibility $\mc{S}_\emptyset$ is correctly accounted for since we included $k = 0$ when indexing the sum in the second line above.
\end{proof}

Now we prove the main result of this section, about estimating $\Sigma[\mu|_{U_i}]$ for open balls $U_i$.

\begin{prop}\label{main empirical covariance estimation}
	Let $\mu$ be a Borel measure supported on a compact subset $K \subset \RR^D$, and let $\Set X = (X_1, \ldots X_m)$ be a $\mu$-i.i.d. sample. Given a radius $r>0$, consider for $1 \leq i \leq m$ the covariances $\hat \Sigma_i := \Sigma[\delta_{\Set{X}_i}|_{U_i}]$ and $\Sigma_i = \Sigma[\MeasRes{\mu}{U_i}]$, where $\Set{X}_i = \{X_j | j \neq i \}$ and $U_i = \Ball_r(X_i)$. Let $\epsilon, \delta, \Proportion > 0$ where we assume\footnote{We lose nothing from this assumption; suppose $\mu, \nu$ are two measures supported on a single ball of radius $r$. Then $\|\Sigma[\mu] - \Sigma[\nu]\|\le 2r^2$ since $\|\Sigma[\mu] - \Sigma[\nu]\| = \sup_{\|x\|=1} x^\top (\EE_{X\sim \mu, Y \sim \nu} XX^\top - YY^\top ) x = \sup_{\|x\|=1} ( \langle X, x \rangle^2 - \langle Y, x\rangle^2 \le 2r^2) \le 2r^2$.} that $\epsilon \le 2r^2$. Then the following holds:
	\[ \frac m{\log m} \ge \frac{1156r^4}{u_0 \epsilon^2} \log \left( \frac{14D \Proportion}{\delta} \right) \implies \Prob\bigg(\max_{i \le \Proportion m} \| \hat\Sigma_{i} - \Sigma_{i} \| \le \epsilon \bigg) \ge 1-\delta \]
	where $u_0 =  \inf_{x \in K} \mu(\Ball_r(x)) > 0$.
\end{prop}
\begin{proof}
    Let $k = \lfloor \Proportion m \rfloor$. Define the set $E_i \subseteq (\RR^{D})^m$ as:
    \[
    E_i := \left\{ \mathbf x = (x_1, \cdots x_m) \mid \left\| \hat\Sigma[\delta_{\mathbf x_i}|_{U_i}] - \Sigma[\MeasRes{\mu}{U_i}] \right\| > \epsilon\right\}.
    \]
    where $\mathbf x_i = \{x_j | j \neq i \}$. By the union bound, symmetry, and Proposition \ref{fixed U}, we then have:
	\begin{align*}
		\mu(E_1 \cup \cdots \cup E_k) &\le  \mu(E_1) + \cdots + \mu(E_k) \\
		&= k \cdot \int  \mu^{k-1} \bigg( \left \{ (x_2, \cdots x_m) | (x_1, x_2, \cdots x_m) \in E_{1} \right \} \bigg) \d \mu(x_1)\\
		& \le k \cdot \int (4D+2) (1-u_x (1-\xi))^{m-1} \d \mu(x)
	\end{align*}
	where $u_x = \mu(\Ball_r(x))$, $\xi = \exp(-\epsilon^2 / 1152r^4)$, and $\mu^{k-1}$ is the product measure on $(\RR^D)^{k-1}$ induced by $\mu$. Since $0 < \xi <1$ and $0 < u_x \le 1$ for any $x$ in the support $K$ of $\mu$, we have that $0 < u_x (1-\xi) < 1$ as well. Letting $u_0 := \inf_{x \in K} u_x$, we have:
	\begin{align}\label{mbound}
	\int (4D+2)k (1-u_x (1-\xi))^{m-1} \d \mu(x) \le (4D+2)k (1-u_0 (1-\xi))^{m-1}
	\end{align}
	Letting right hand side of \eqref{mbound} to be $\le \delta$, we get the condition:
	\begin{align} 
	    & (4D+2)k (1-u_0 (1-\xi))^{m-1} \le \delta \nonumber \\
	    \iff & \frac{-1}{\log \left( 1 - u_0(1-\xi) \right)} \cdot \log \left( \frac{(4D+2)k}{\delta} \right) \le m-1 \label{mbound2}
	\end{align}
	To produce a simpler lower bound for $m$, we calculate:
	\begin{align*}
	    \frac{-1}{\log \left( 1 - u_0(1-\xi) \right)} \le \frac1{u_0}\left( \frac{1152r^4}{\epsilon^2} + 1\right) - \frac12 \le \frac1{u_0} \cdot \frac{1156r^4}{\epsilon^2} - \frac12
	\end{align*}
	where the first inequality is due to Lemma \ref{calculus lemma 2}, and the second inequality follows from the assumption that $\epsilon^2 \le 4r^4$.\footnote{By similar reasoning, the left hand side of \eqref{mbound2} is at least $\frac1{u_0}(1150r^4/\epsilon^2)$, so that this sufficient condition doesn't weaken the bound much.} Using the fact that $\log((4D+2)/\delta) \ge 2$ and Lemma \ref{simple log lemma}, we obtain the claimed sufficient condition for \eqref{mbound2}:
	\[ \frac{1156r^4}{u_0 \epsilon^2} \log \left( \frac{14 D \Proportion}{\delta} \right)\le \frac m{\log m} \]
	
	To establish that $u_0 > 0$, consider the covering of $K$ by balls of radius $r/2$. Since $K$ is compact, it admits a subcover $\{\Ball_{r/2}(x) ~ \mid ~ x \in J\}$, with $J$ a finite set. Thus, every $x \in K$ admits a $y \in J$ satisfying $x \in \Ball_{r/2}(y)$. Triangle inequality guarantees that $\Ball_{r/2}(y) \subseteq \Ball_r(x)$, so that $\mu(\Ball_{r/2}(y)) \le \mu(\Ball_r(x))$ and hence
	$\inf_{y \in J}  \mu(\Ball_{r/2}(y)) \le \inf_{x \in K} \mu(\Ball_r(x))$. Since the left hand side is an infimum over a finite set of strictly positive numbers, it is also strictly positive and we have $u_0 > 0$ as desired.
\end{proof}

\section{Lipschitz property of covariance matrix}\label{chap3}
Our goal in this section is to outline sufficient conditions under which the assignment $\mu \mapsto \Sigma[\mu]$ becomes a Lipschitz function with respect to the Wasserstein distance \cite{villani_transport} on its domain, defined as follows. Let $(M,\d_M)$ be a Polish metric space equipped with probability measures $\mu$ and $\nu$. For each $p \ge 1$, the {\em $p$-Wasserstein distance} between $\mu$ and $\nu$ equals
\[
\Wass_p(\mu, \nu) := \left( \inf_{\gamma \in \Pi(\mu, \nu)} \int_{M \times M} \d_M(x,y)^p  \d \gamma(x,y) \right)^{1/p} 
\]
where $\Pi(\mu, \nu)$ is the set of measures on $M \times M$ with marginals equal to $\mu$ and $\nu$. Note that whenever $1 \le p \le q$, we have $\Wass_p(\mu, \nu) \le \Wass_q(\mu, \nu)$ by the power mean inequality. Throughout this section, we use the notation $X \sim \mu$ and $Y \sim \nu$, whenever probability distributions $\mu, \nu$ are defined.

\begin{lem}\label{centering lipschitz}
    Given Borel probability measures $\mu, \nu$ valued in $\RR^D$, define $\tilde \mu = \Law(X - \EE X)$ and similarly $\tilde \nu$. Then for each $p \ge 1$,
    \begin{enumerate}
        \item $\| \EE X - \EE Y \| \le \Wass_p(\mu, \nu)$
        \item $\Wass_p(\tilde \mu, \tilde \nu) \le 2 \cdot \Wass_p(\mu, \nu)$
    \end{enumerate}
\end{lem}
\begin{proof}
Defining $x_0 := \EE X$ and $y_0 := \EE Y$, we have
	\begin{align*}
	\| x_0 - y_0 \| 
	&= \left \| \int_{\RR^D} \int_{\RR^D} (x-y) \d \mu(x) \d \nu(y) \right\| \\
	&= \left \| \int_{\RR^D \times \RR^D} (x - y) \d \gamma (x, y) \right \| \text{, for any $\gamma \in \Pi(\mu, \nu)$} \\
	&= \inf_{\gamma \in \Pi(\mu, \nu)} \left \| \int_{\RR^D \times \RR^D} (x - y) \d \gamma (x, y) \right \| \\
	&\le \inf_{\gamma \in \Pi(\mu, \nu)} \int_{\RR^D \times \RR^D} \| x - y\|  \d \gamma (x, y) \\
	&= \Wass_1 (\mu, \nu)
\end{align*}
Noting that $W_1(\mu,\nu) \le W_p(\mu, \nu)$ for any $p\ge1$, we get the first claim. For the second claim,
\begin{align*}
	 \Wass_p(\tilde \mu, \tilde \nu)^p 
	&= \inf_{\gamma \in \Pi(\mu, \nu)} \int_{\RR^D \times \RR^D} \| (x - x_0) - (y - y_0) \|^p \d \gamma(x,y) \\
	&= 2^p \cdot \inf_{\gamma \in \Pi(\mu, \nu)}  \int_{\RR^D \times \RR^D} \left( \frac{\| x - y\| + \| x_0 - y_0 \| }{2} \right )^p \d \gamma(x,y) \\
	&\le 2^p \cdot \inf_{\gamma \in \Pi(\mu, \nu)} \int_{\RR^D \times \RR^D} \frac { \| x - y \|^p + \|x_0 - y_0\|^p }2 \d \gamma(x, y) \\
	&= 2^{p-1} (\Wass_p (\mu, \nu)^p + \| x_0 - y_0 \|^p ) \\
	&\le  2^p \cdot \Wass_p (\mu, \nu)^p
\end{align*}
where the first inequality is the power mean inequality, and the second inequality follows from the first claim.
\end{proof}

\begin{lem}\label{1dim variance lipschitz}
    For probability measures $\mu, \nu$ defined on $\RR$ and supports contained the interval $[-R, +R]$, we have the $2R$-Lipschitz relation for all $p\ge 1$:
    \[ \EE[X^2] - \EE[Y^2] \le 2R \cdot \Wass_p (\mu, \nu) \]
\end{lem}
\begin{proof}
    Since $\Wass_p$ is increasing in $p$, it suffices to prove the assertion for $p=1$.
    \begin{align*}
     \EE[X^2] - \EE[Y^2] 
    	&= \int_\RR  \int_\RR (x^2 - y^2) \d \mu(x) \d \nu(y) \\
    	&= \int_{\RR \times \RR}  (x^2 - y^2) \d\gamma(x,y) \text{, for any } \gamma \in \Pi(\mu, \nu) \\
    	&\le 2R \cdot \inf_{\gamma \in \Pi(\mu, \nu)} \int_{\RR \times \RR} |x- y| \d \gamma(x,y) \\
    	&= 2R \cdot \Wass_1 (\mu, \nu)
    \end{align*}
    where the only inequality above follows from the fact that the derivative of $f(x) = x^2$ is bounded by $2R$ if $x \in [-R, +R]$.
\end{proof}

\begin{prop}\label{covariance lipschitz}
	Suppose $\mu, \nu$ are probability measures on $\RR^D$ such that each measure comes with a ball of radius $r$ that contains the support of the measure. Then for $p \ge 1$, we have the following Lipschitz property:
	\[ \|\Sigma[\mu] - \Sigma[\nu] \| \le 4 r \cdot \Wass_p(\tilde \mu, \tilde \nu) \le 8 r \cdot \Wass_p(\mu, \nu) \]
	where $\tilde \mu= \Law(X - \EE X)$.
\end{prop}
\begin{proof}
	We assume that $r=1$, since the case for general $r$ follows by scaling: $r$ affects the covariance matrix on the order of $r^2$ and the Wasserstein distance on the order of $r$. Also, the second inequality follows from the first by Lemma \ref{centering lipschitz}, so it suffices to show the first inequality. Since we are then working with $\tilde \mu$ and $\tilde \nu$ and since covariance matrix is invariant under translation, we may rewrite $\mu = \tilde\mu$ and $\nu = \tilde\nu$ and assume that $\mu, \nu$ have zero means. We may also assume that both $\on{supp} \mu$  and $\on{supp} \nu $ are contained within $\Ball_2(0)$ by the triangle inequality; there is a ball $\Ball_1(x)$ of radius 1 containing $\on{supp}\mu$, so that by triangle inequality, $\on{supp} \mu \subseteq \Ball_1(x) \subseteq \Ball_2 (0)$.
	
	Denoting $S := \Sigma[\mu] - \Sigma[\nu]$, it is a real symmetric matrix and we may diagonalise it as $S = U \Lambda U^\top$. $U=[u_1, \ldots u_D]$ is orthogonal and $\Lambda$ is a diagonal matrix with entries $\lambda_1 \ge \cdots \ge \lambda_D$. The operator norm of $S$ is $\max_i |\lambda_i|$, which can be written as:
    \begin{align*}
        \|S\| &= \max_i |\lambda_i| = \max_i \big| (U^\top S U)_{i,i} \big| \\
        &= \max_i \big| \EE[ U^\top X X^\top U ]_{i,i} - \EE[ U^\top Y Y^\top U ]_{i,i} \big| \\
        &= \max_i \big| \EE (U^\top X)_i^2 - \EE (U^\top Y)_i^2  \big|
    \end{align*}
    where $A_{i,i}$ refers to the $(i,i)$th entry of a matrix $A$ and $w_i$ refers to the $i$st entry of a vector $w$. Now we are done by the following that holds for all $i$:
    \begin{align*}
     \EE (U^\top X)_i^2 - \EE (U^\top Y)_i^2 
        &\le  4 \Wass_1((U^\top \mu)_i, (U^\top \nu)_i ) \\
        &\le 4 \Wass_1(U^\top \mu, U^\top \nu ) \\
        &=  4 \Wass_1(\mu, \nu)
    \end{align*}
    where $U^\top \mu = \Law(U^\top X)$ and $(U^\top \mu)_i$ denotes the marginal of $U^\top \mu$ at its $i$th coordinate. The first inequality is Lemma \ref{1dim variance lipschitz} with $2R = 4$. The second inequality is a general fact that applies to the Wasserstein distances between marginals. The last equality follows from the fact that the Wasserstein distance is invariant with respect to isometry applied simultaneously to the two measures. Finally, multiplying by the Lipschitz constant $2$ for the non-centered measures, we get the Lipschitz constant $8$. The inequality for other $p$ follows since $\Wass_p$ is increasing in $p$.
\end{proof}

\section{Wasserstein bound for Flattening a Measure on Manifold}\label{chap4}

In this section, we quantify the extent to which a probability distribution valued near a manifold approximates the uniform distribution over a tangential disk, using the Wasserstein distance. We first define the measure of interest using a probability density function, Hausdorff measure, and a noise term.

\begin{defn}
	Given a metric space and a positive integer $d$, denote by $\mc H^d$ the $d$-dimensional Hausdorff measure \cite{simon_gmt} on the metric space:
	\begin{align*}
		\mc H^d(U) =& \lim_{\delta \downarrow 0} \mc H^d_\delta(U), \quad \mc H^d_\delta(U) = \frac{\omega_d}{2^d} \inf_{\substack{\on{diam}(C_j) < \delta \\ U \subseteq \cup C_j}} \left( \sum_{j=1}^\infty \on{diam}(C_j)^d \right)
	\end{align*}
	where $\omega_d := \frac{\pi^{d/2}}{\Gamma(\frac d2 + 1)}$. Given a Borel set $U \subseteq \RR^D$ with a finite, nonzero real $d$-dimensional Hausdorff measure $\mc H^d(U) \in (0, \infty)$, denote by $\Unif(U)$ the $d$-dimensional uniform probability measure over $U$ with respect to $\mc H^d$; for each $V$,
	\[ \Unif(U) := \mc H^d|_U, \text{ i.e. } \Unif(U)(V) = \frac{\mc H^d(U \cap V)}{\mc H^d(U)}\]
\end{defn}

\begin{defn}\label{our measures defn}
    Suppose $M$ is a $d$-dimensional smooth compact manifold with a smooth embedding into $\RR^D$ and $\varphi : M \rightarrow \RR^+$ is a continuous function satisfying $\int_M \varphi \d \mc H^d = 1$. Let $\mu_0$ be the Borel probability measure given by defining for each open $U\subseteq \RR^D$ the following:
    \[\mu_0(U) = \int_{U \cap M} \varphi \d \mc H^d \]
    Let $s\ge0$ be a constant, $X \sim \mu_0$ and let $Y$ be a random variable valued in $\RR^D$ with bounded norm $\|Y\| \le s$. Here $X$ and $Y$ are \textit{not} assumed to be independent. Define \[\mu := \Law(X+Y)\]
    Then $\OurMeasures(M, s)$ is defined as the set of all such pairs $(\mu_0, \mu)$, given $M$ and $s$.
\end{defn}

The following are notions from differential geometry relevant to us.

\begin{defn}\label{reach defn}
    For each compact Riemannian manifold $M \subset \RR^D$,
    
    \begin{enumerate}
    \item For each $x, y \in M$, let $\d_M(x,y)$ be the length of the shortest geodesic connecting $x$ and $y$.\footnote{Equivalently, $\d_M(x,y)$ be the infimum of lengths of all piecewise regular curves that connect $x$ and $y$. This follows from the Hopf-Rinow Theorem; see Corollary 6.21 and 6.22 in \cite{lee_riemannian}.}
    
    \item The \textit{reach} $\tau$ of $M$ is the supremum of $t \ge 0$ satisfying the following: If $x\in \RR^D$ satisfies $\d_{\RR^D}(x, M) \le t$, then there is a unique point $x_\perp \in M$ such that $\d_{\RR^D}(x, x_\perp) = \d_{\RR^D}(x, M)$. Here, $\d_{\RR^D}(x,y) = \|x-y\|$ is the Euclidean distance on $\RR^D$, and $\d_{\RR^D}(x, M) = \inf_{y \in M} \d_{\RR^D}(x,y)$. 
    
    \item For each point $x \in M$, we denote by $\ExpBall_r \subseteq T_x M$ the open ball of radius $r$ around $0 \in T_x M$, while the notation $\Ball_r(x) \subseteq \RR^D$ is reserved for the (usual) open ball of radius $r$ around $x \in \RR^D$. 
    
    \item Given $x \in M$, the \textit{exponential map} $\exp_x$ sends each $v \in T_x M$ to the endpoint of the unique geodesic on $M$ starting at $x$ with the initial velocity of $v$.
    \end{enumerate}
\end{defn}

We remark that $1/\tau$ is an upper bound of the acceleration of geodesics on $M$ in the ambient space $\RR^D \supset M$. The following is the main result of this section.

\begin{prop}\label{main transport}
    Let $(\mu_0, \mu) \in \OurMeasures(M, s)$ where $M \subseteq \RR^D$ is a compact smoothly embedded $d$-dimensional manifold with reach $\tau$ and $s \ge 0$. Let $x \in \on{supp} \mu$, let $x_\perp$ be any point in $\Ball_s(x) \cap M$, and let $r$ be a number satisfying the conditions $2s \le r \le (\sqrt{2}-1)\tau - 2s$ and $r \le \tau/(2\sqrt{2} d)$. Then the following holds for any $p\ge1$:
    \begin{align*}
        & \Wass_p(\nu, \tilde\nu) \le \tau \cdot Q\left(\frac{r}\tau, \frac{s}\tau\right)\\
        \text{where } & \nu := \MeasRes{\mu}{\Ball_r(x)} \text{, and } \tilde\nu := \Unif( \Ball_{r}(x_{\perp}) \cap T_{x_\perp} M )
    \end{align*}
    where $Q$ is given by:
    \begin{align*}
        Q(\rho, \sigma) = & 3 \sigma + (\rho+2\sigma)^2 + 2\rho (1 - \Omega^d) \frac1{\Phi} \varphi_{\max} \bigg( 1 + 4\sqrt 2 d \rho \bigg)  \\
         & + \left( 
        \frac1{\Phi} \left(\varphi_{\max} - \varphi_{\min} \right) ( 1 + 4\sqrt{2} d \rho ) + 4\sqrt 2 d \rho \right) \cdot 2\rho + \frac14 \rho^3
    \end{align*}
    where $\varphi_{\max}, \varphi_{\min}$ are extrema of $\varphi$ taken over $\Ball_{r+2s}(x_\perp)$ and
    \begin{align*}
        \Phi = \frac{\mu_0(\Pi^{-1} \Ball_-^\circ )}{\omega_d r_-^d}, \> \Omega = \frac{r_-}{r_+}, \> r_- = r\bigg(1-\frac{r^2}{4\tau^2}\bigg)-2s, \> r_+ = r + 2s
    \end{align*}
    and $\Pi$ is the projection map to $T_{x_\perp} M$.
\end{prop}
\begin{proof}
    We use the following multi-step transportation plan (see Figure \ref{main transport fig1}), from $\nu_0 := \nu$, going through $\nu_1, \nu_2, \nu_3, \nu_4$ which we define below and finally reaching $\nu_5 := \tilde \nu$. Informally, these steps can be summarized as
    \begin{enumerate}
        \item Perform a naive denoising on $\nu_0$ to get $\nu_1$
        \item Apply projection to get $\nu_2$
        \item Fold in the portion of $\nu_2$ on the outer rim to the inside to get $\nu_3$
        \item Flatten out the nonuniformity and get $\nu_4$.
        \item Rescale radius uniformly to get $\nu_5$.
    \end{enumerate}
    
    \begin{figure}\label{main transport fig1}
        \centering
        \includegraphics[scale=0.33]{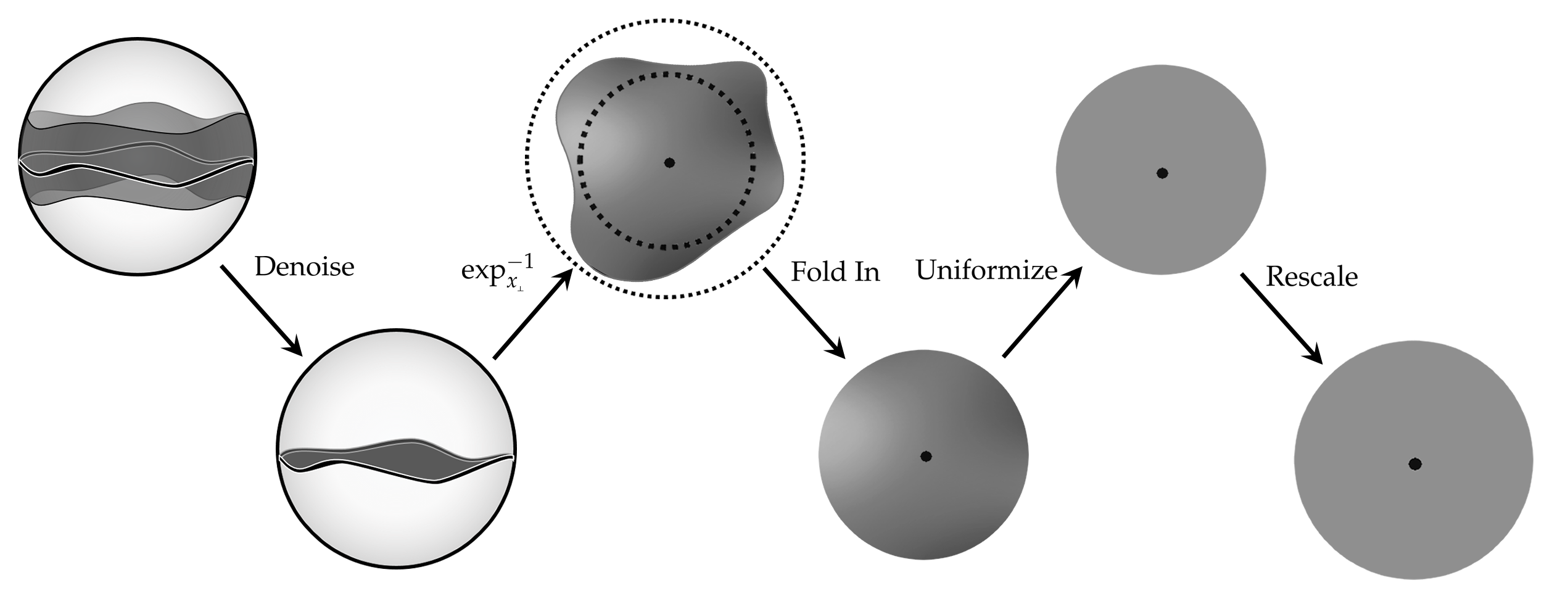}
        \caption{An overview of the transportation plan in the proof of Proposition \ref{main transport}. The last four sub-diagrams take place on the tangent space. Nonuniform shadings in the 3rd, 4th sub-diagrams indicate nonuniform probability distribution.}
     \end{figure}

    \textbf{Step 1.} Suppose that $X \sim \mu_0$ and $(X+Y) \sim \mu$. We define $\nu_1 := \Law(X \mid X+Y \in \Ball_r(x))$  and define the transportation plan $\nu_{01}$ by $\nu_{01} := \Law((X+Y, X) \mid X+Y \in \Ball_r(x))$, whose marginals are $\nu_0$ and $\nu_1$. Thus for each open $U \subseteq \RR^D$, we have
        \begin{align}
        \nu_1(U) =& \Prob(X \in U \mid X+Y \in \Ball_r(x)) \nonumber \\
        =& \frac1{u} \Prob(X \in U \text{ and } X+Y \in \Ball_r(x)) \nonumber \\
        \text{where } u =& \mu(\Ball_r(x)) \label{nu1 calculation} 
        \end{align}
    where $u = \mu(\Ball_r(x)) = \Prob(X+Y \in \Ball_r(x))$, which follows by the definition of $\mu$. The transportation cost is bounded as 
    \[ \Wass_p(\nu_0, \nu_1) \le \EE_{(X+Y, X) \sim \nu_{01}} \|(X+Y) - X\| \le s\]
    Note that by the assumption $x \in \on{supp} \mu$, we have $u > 0$ and thus we are not conditioning on the null event. 
    
    By Equation \eqref{nu1 calculation}, $\nu_1$ is well understood in regions where the condition $X+Y \in \Ball_r(x)$ either always or never holds. If $X \in \Ball_{r-s}(x)$, then since $\|Y\| \le s$, the triangle inequality implies $X+Y \in \Ball_r(x)$. Similarly if $X \notin \Ball_{r+s}(x)$, then $X+Y \notin \Ball_r(x)$. By also noting that $\|x-x_\perp\| \le s$, the triangle inequality once again implies $\Ball_{r-2s}(x_\perp) \subseteq \Ball_{r-s}(x)$ and $\Ball_{r+s}(x) \subseteq \Ball_{r+2s}(x_\perp)$. Applying Equation \eqref{nu1 calculation}, we get the following:
    \begin{align}
        & \nu_1(U) \le \frac{\mu_0(U)}{u} && \text{for any $U$} \nonumber \\
        & \nu_1(U) = \frac{\mu_0(U)}{u} && \text{for $U \subseteq \Ball_{r-2s}(x_\perp)$} \nonumber \\
        & \nu_1(U) = 0 && \text{for $U \subseteq \Ball_{r+2s}(x_\perp)^\Complement$} \label{nu1 proportionality}
    \end{align}
    where $A^\Complement$ denotes the complement of a set $A$. Note that $\mu(\Ball_r(x))$ is a constant, since we fixed $x$.

    \begin{figure}\label{main transport fig2}
        \centering
        \includegraphics[scale=0.35]{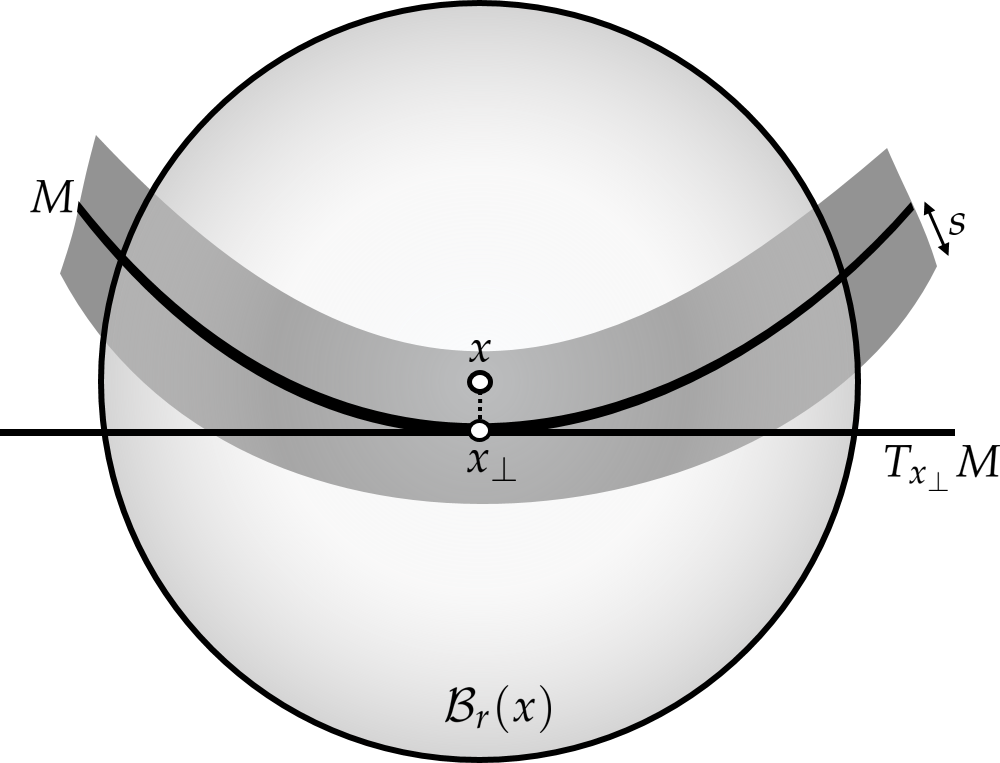}
        \caption{Measure $\mu$ and its restriction $\MeasRes{\mu}{\Ball_r(x)}$, where $x \in \RR^D$ and $x_\perp \in M$.}
     \end{figure}

    \textbf{Step 2.} We define $\nu_2$ by pushing forward $\nu_1$ along the projection map to the tangent space, and we must do it where the map is invertible. By Lemma \ref{jacobian lemmas}, we know that the projection map is a diffeomorphism. Now, denote     $\Pi$ by the projection map to $T_{x_\perp} M$. Then,
    \begin{align}
        \Ball_-^\circ \subseteq \Pi(\Ball_- \cap M), \quad \Pi(\Ball_+ \cap M) \subseteq \Ball_+^\circ \label{nice inclusions}
    \end{align}
    where, denoting $\ExpBall_{r}$ by the open ball of radius $r$ in $T_{x_\perp} M$ centered at $0$, \footnote{Note that $(r-2s)(1-(r-2s)^2/4\tau^2) \le (r-2s)(1-r^2/4\tau^2) = r(1-r^2/4\tau^2) - 2s(1-r^2/4\tau^2) \le r(1-r^2/4\tau^2) - 2s$}
    \begin{align*}
        \Ball_- = \Ball_{r-2s}(x_\perp), &\quad \Ball_+ = \Ball_{r+2s}(x_\perp) \\
        \Ball_-^\circ = \ExpBall_{r_-}, &\quad \Ball_+^\circ = \ExpBall_{r_+} \\
        r_- = r\bigg(1-\frac{r^2}{4\tau^2}\bigg)-2s, &\quad r_+ = r+2s
    \end{align*}
    The transportation plan is the application of Lemma \ref{wass1} to the pushforward along $\Pi$. In performing the transportation, we regard the tangent space as embedded: $T_{x_\perp} M \subseteq \RR^D$ so that the transportation happens in the ambient space $\RR^D$. By the last result mentioned in Lemma \ref{geodesic bound in ball}, the transportation cost then is bounded as:
        \[ \Wass_p(\nu_1, \nu_2) \le \frac{(r+2s)^2}{\tau} \]
    Thus by Equations \eqref{nu1 proportionality} and \eqref{nice inclusions},
    \begin{align}
        & \nu_2(U) \le \frac{\mu_0(\Pi^{-1} U)}{u} && \text{for $U \subseteq \ExpBall_+ $} \nonumber \\
        & \nu_2(U) = \frac{\mu_0(\Pi^{-1} U)}{u} && \text{for $U \subseteq \ExpBall_- $} \nonumber \\
        & \nu_2(U) = 0 && \text{for $U \subseteq (\ExpBall_+)^\Complement$}
        \label{nu2 proportionality}
    \end{align}
    Meanwhile, we can evaluate $\mu_0(U)$ when $U \subseteq \ExpBall_+$ explicitly using the area formula from geometric measure theory\footnote{See for example \cite{federer_gmt} for a standard reference in geometric measure theory}, which is a generalization of chain rule:
    \begin{align} \label{eq: mu0 formula}
        \mu_0(\Pi^{-1} U) = \int_{\Pi^{-1} U} \varphi \d \mc H^d = \int_{U} \varphi(\Pi^{-1} y) \on{J} \Pi^{-1} (y) \d y
    \end{align}
    Here, $\on{J} f$ denotes the Jacobian of a function $f$ and $\d y$ is the $d$-dimensional Lebesgue measure. Thus,
    \begin{align}
        & \nu_2(U) \le \frac1{u} \int_{U} \varphi(\Pi^{-1} y) \on{J} \Pi^{-1} (y) \d y && \text{for $U \subseteq \ExpBall_+$ } \nonumber \\
        & \nu_2(U) = \frac1{u} \int_{U} \varphi(\Pi^{-1} y) \on{J} \Pi^{-1} (y) \d y && \text{for } U \subseteq \ExpBall_- \nonumber \\
        & \nu_2(U) = 0 && \text{for $U \subseteq (\ExpBall_+)^\Complement$} \label{nu2 calculation}
    \end{align}

    \textbf{Step 3.} We saw that $\nu_2$ can be written in terms of $\mu_0$ inside radius $r_-$ and vanishes outside radius $r_+$. The annular region between the two radii is harder to understand since it is where curvature and noise interact, as indicated by Equation \eqref{nu1 calculation}. In Step 3 we remove this annular region, so that we only need to deal with $\nu_2$ restricted to $\ExpBall_-$. We decompose $\nu_2$ as $\nu_2 = \nu_2^- + \nu_2^+$, where we define for each Borel set $U \subseteq T_{x_\perp} M$ the following:
    \begin{align*}
    & \nu_2^-(U) := \nu_2(U \cap \ExpBall_-) \\
    & \nu_2^+(U) := \nu_2(U \cap (\ExpBall_+ - \ExpBall_-))
    \end{align*}
    Define
    \begin{align*}
        \nu_3 :=& \nutwomass^{-1} \nu_2^- \\
        \nutwomass :=& \nu_2^-(T_{x_\perp} M)
    \end{align*}
    The transportation plan is to: (a) transport $\nu_2^+$ to the Dirac delta distribution centered at $0 \in T_x M$ and (b) transport this Dirac delta distribution back to $\frac{1-\nutwomass}{\nutwomass} \nu_2^-$. By Lemma \ref{wass1_5}, we have the bound:
    \[ \Wass_p(\nu_2, \nu_3) \le (r_+ + r_-)(1-\nutwomass) \le 2r(1-\nutwomass) \]
    since the first part of this transportation moves by distance at most $r_+$, the second part moves by at most $r_-$, and the total mass to move is $(1-\nutwomass)$. Equation \eqref{nu2 calculation} carries over since $\nu_3$ and $\nu_2^-$ are proportional; for each open $U \subseteq T_{x_\perp} M$, 
    \begin{align}
        & \nu_3(U) = \frac1{u \nutwomass} \int_{U \cap \ExpBall_-} \varphi(\Pi^{-1} y) \on{J} \Pi^{-1} (y) \d y \label{nu3 calculation}
    \end{align}

    \textbf{Step 4.} We flatten out the non-uniformity in $\nu_3$. As in Equation \eqref{nu3 calculation} above, $\nu_3$ is given by the probability density function $\psi(y) := \varphi(\Pi^{-1} y) \on{J} \Pi^{-1} (y)$ times a constant. Defining $\nu_4 = \Unif(\ExpBall_-)$, we can directly apply Lemma \ref{wass2}:
    \begin{align*}
        \Wass_p(\nu_3, \nu_4) \le \frac{\omega_d r_-^d }{u \nutwomass } \cdot (\psi_{\max} - \psi_{\min}) \cdot 2r_-
    \end{align*}
    where the factor $\omega_d r_-^d$ is needed to rescale the Lebesgue measure $\d y$ in Equation \eqref{nu3 calculation} into $\widetilde{\d y} = {\d y}/(\omega_d r_-^d)$ so that $\int_{\ExpBall_-} \widetilde{\d y}= 1$, so that Lemma \ref{wass2} can be applied. In the above, extrema of $\psi$ are taken over $\ExpBall_-$. Since $\psi$ is the product of $\varphi$ and the Jacobian, the variation $\psi_{\max} - \psi_{\min}$ can be controlled with the triangle inequality as follows:
    \[ | \psi_{\max} - \psi_{\min} |  \le  \left( \varphi_{\max} - \varphi_{\min} \right) J_+ + \varphi_{\min} (J_+ - J_-) \]
    Here the extrema of $\varphi$ are taken over the geodesic ball $\Pi^{-1} \ExpBall_-$. By Proposition \ref{jacobian lemmas}, we see that:
    \begin{align}
        & J_- \le \on{J} \Pi^- \le J_+ \nonumber \\
        \text{where } & J_- = 1, \> J_+ = \bigg( 1-\frac{\sqrt{2}r}{\tau} \bigg)^{-d} \label{jacobian bound eqn}
    \end{align}
    We furthermore note that, by Equation \ref{nu2 calculation},
    \begin{align*}
        & u \nutwomass = \int_{\ExpBall_-} \varphi(\Pi^{-1} y) \on{J} \Pi^{-1}(y) \d y \ge \omega_d r_-^d J_- \varphi_{\min} \\
        & \implies \varphi_{\min} \le \frac{u \nutwomass}{\omega_d r_-^d } \cdot \frac1{J_-}
    \end{align*}
    Thus the transportation cost is bounded as:\footnote{We note at this point that the extrema of $\varphi$ may be taken over $\Ball_{r+2s}(x_\perp)$ instead, since $\Ball_{r+2s}(x_\perp) \supseteq \Pi^{-1}(\ExpBall_-)$. This relaxation is done for a compatibility with another extrema of $\varphi$ taken later.}
    \begin{align*}
        \Wass_p(\nu_3, \nu_4) \le \left( 
    \frac{\omega_d r_-^d }{u \nutwomass } \left(\varphi_{\max} - \varphi_{\min} \right) J_+ + \frac{J_+ - J_-}{J_-} \right) \cdot 2r_-
    \end{align*}

    \textbf{Step 5.} Here we simply rescale $\ExpBall_-$ from radius $r_-$ to $r$ radially, which multiplies the associated probability density function by a constant factor (Lemma \ref{radial scaling jacobian}), so that we get another uniform distribution. By Lemma \ref{wass1}, the transportation cost is bounded by
    \[ \Wass_p(\nu_4, \nu_5) \le r - r_- = \frac{r^3}{4\tau^2} + 2s \]
    
    \textbf{The Total Bound.} Collecting the bounds\footnote{We plug in the definition $J_-=1$, and we also use a slight abuse of notation and identify $\nu_k$ with $\iota_* \nu_k$ for $k=2, \ldots 5$, where $\iota: T_{x_\perp}M \hookrightarrow \RR^D$ is the inclusion of tangent space. This is not a problem, since generally $\Wass_p(\iota_* \mu_1, \iota_* \mu_2) \le \Wass_p(\mu_1, \mu_2)$ holds for any measures $\mu_1, \mu_2$ on $T_{x_\perp} M$.}, we get:
    \begin{align}
        & \Wass_p(\nu_0, \nu_5) \nonumber \\
        \le & \Wass_p(\nu_0, \nu_1) + \Wass_p(\nu_1, \nu_2) + \Wass_p(\nu_2, \nu_3) + \Wass_p(\nu_3, \nu_4) + \Wass_p(\nu_4, \nu_5) \nonumber \\
        \le & s + \frac{(r+2s)^2}{\tau} + 2r (1-\nutwomass) + \nonumber \\
        & + \left( 
        \frac{\omega_d r_-^d }{u \nutwomass } \left(\varphi_{\max} - \varphi_{\min} \right) J_+ + (J_+-1) \right) \cdot 2r + \bigg(\frac{r^3}{4\tau^2} + 2s\bigg) \label{initial wass bound}
    \end{align}
    Using Equations \eqref{nu2 proportionality}, \eqref{nu2 calculation} and \eqref{jacobian bound eqn}, we obtain the following bounds:
    \begin{align*}
        & u \nutwomass = \mu_0(\Pi^{-1}\ExpBall_-) \le \varphi_{\max} J_+ \omega_d r_-^d \\
        & u (1-\nutwomass) \le \mu_0(\Pi^{-1}(\ExpBall_+ - \ExpBall_-)) \le \varphi_{\max} J_+ \omega_d (r_+^d - r_-^d)
    \end{align*}
    where $\varphi_{\max}$ is the maximum of $\varphi$ taken over $\Ball_{r+2s}(x_\perp)$.\footnote{See Equation \ref{nice inclusions}.} Combining these, we get:
    \begin{align*}
        & \frac{1-\nutwomass}{\nutwomass} = \frac{u(1-\nutwomass)}{u\nutwomass} \le \frac{ \varphi_{\max} J_+ \omega_d (r_+^d - r_-^d) } {u\nutwomass}  = \Phi' (\Omega^{-d}-1) \\
        \text{with } & \Omega = \frac{r_-}{r_+}, \Phi' = \frac{\varphi_{\max} J_+ \omega_d r_-^d }{u \nutwomass} \ge 1
    \end{align*}
    We can bound $\smallint \nu_2^{\on{out}}$ using the above, as follows:
    \begin{align*}
        & 1-\nutwomass = \left(1 + \frac{\nutwomass}{1-\nutwomass} \right)^{-1} \le \left(1 + \frac1{\Phi'(\Omega^{-d} - 1)} \right)^{-1} \le \Phi'(1 - \Omega^d)
    \end{align*}
    where the first inequality holds by plugging in the upper bound for $(1-\nutwomass) / \nutwomass$, and the second inequality holds since $\Phi' \ge 1$. Plugging these into Equation \eqref{initial wass bound}, we get that 
    \begin{align*}
        \Wass_p(\nu_0, \nu_5) \le & s + \frac{(r+2s)^2}{\tau} + 2r (1 - \Omega^d) \varphi_{\max} J_+ \frac{\omega_d r_-^d }{ u\nutwomass } \\
         + &\left( 
        \frac{\omega_d r_-^d }{u \nutwomass } \left(\varphi_{\max} - \varphi_{\min} \right) J_+ + (J_+-1) \right) \cdot 2r + \bigg( \frac{r^3}{4\tau^2} + 2s \bigg) 
    \end{align*}
    We bound $J_+$ using Lemma \ref{calculus lemma 3}. By applying the assumption $r \le \tau/(2\sqrt{2} d)$, we see that the lemma applies with $c=1/2$:
    \begin{align*}
        \bigg( 1-\frac{\sqrt 2 r}\tau \bigg)^{-d} \le 1 + \frac{4\sqrt 2 d \cdot r}{\tau}
    \end{align*}
    Applying this to the above bound on $\Wass_p(\nu_0, \nu_5)$ and also plugging in $\rho = r/\tau, \sigma = s/\tau$, we obtain the $Q(\sigma, \tau)$ expression that was claimed in the beginning. 
\end{proof}

\begin{cor}\label{main transport simplified}
    In Proposition \ref{main transport}, suppose that we additionally assume that there exists $\alpha$ such that the following Lipschitz continuity holds for every $x, y \in M$:
    \begin{align*}
        & \|\varphi(x) - \varphi(y)\| \le \alpha \cdot \d_M(x,y)
    \end{align*}
    Suppose we also assume $s \le r^2/(2\tau)$. Then we have the following quadratic bound:
    \[ \Wass(\nu, \tilde \nu) \le Q_2 \cdot \tau \rho^2 \]
    where $Q_2$ is defined as:
    \[Q_2 = \QtwoDefinition \]
\end{cor}
\begin{proof}
We use the notation $\rho = r/\tau, \sigma = s/\tau$. Firstly by the assumption $\sigma \le \rho^2/2$,
\begin{align*}
    \Omega =& \frac{\rho - \rho^3/4 - 2\sigma}{\rho + 2\sigma} \ge \frac{1 - \rho^2/4 - \rho}{1 + \rho} \ge \frac{1-c\rho}{1+c\rho} \text{, where } c = \frac 98
\end{align*}
Then, assuming $\rho \in [0, 8/9]$, Lemma \ref{calculus lemma 4} says:
\begin{align*}
    1-\Omega^d \le & 1 - \frac{(1-c\rho)^d}{(1+c\rho)^d} \le 2dc \cdot \rho
\end{align*}
By the Lipschitz condition and the noise bound,
\begin{align*}
    Q(\rho, \sigma) = & 3 \sigma + (\rho+2\sigma)^2 + 2\rho (1 - \Omega^d) \frac1{\Phi} \varphi_{\max} \bigg( 1 + 4\sqrt 2 d \rho \bigg)  \\
     & + \left( 
        \frac1{\Phi} \left(\varphi_{\max} - \varphi_{\min} \right) ( 1 + 4\sqrt{2} d \rho ) + 4\sqrt 2 d \rho \right) \cdot 2\rho + \frac14 \rho^3 \\
    \le & \frac 32 \rho^2 + (1+\rho)^2\rho^2 + \frac1\Phi \varphi_{\max} 4dc   \bigg(1+4\sqrt{2}d\rho \bigg) \rho^2 \\
    & + \left( \frac1\Phi (2(r+2s) \alpha )(1+4\sqrt{2}d\rho) + 4\sqrt{2}d \right) \cdot 2\rho^2 + \frac14 \rho^3
\end{align*}
where the Lipschitz relation is applied to bound $\varphi_{\max}-\varphi_{\min} \le 2(r+2s)\alpha$ by using \textit{two} radial geodesics of length $\le r_+=r+2s$ in the unit ball of radius $r_+$ in the tangent space $T_{x_\perp} M$. Factoring out $\rho^2$ and plugging back in the definition $c = \frac 98$, we get:
\[ \frac1{\rho^2} Q(\rho, \sigma) \le 
\bigg( 8\sqrt{2}d + \frac52 + \frac94\rho + \rho^2 \bigg) + \frac{\varphi_{\max}}{\Phi} \frac{9}{2} d(1+4\sqrt{2}d\rho) + \frac{4(\rho+\rho^2)\alpha \tau }{\Phi} (1+4\sqrt{2}d\rho) \]
Using the assumption $\rho \le \frac1{2\sqrt{2}d}$, we get the bounds $1 + 4\sqrt{2}d\rho \le 3$, and $9\rho/4 + \rho^2 \le 1$, and $\rho + \rho^2 \le \frac12$. We obtain the claimed bound by plugging them in.
\end{proof}

\section{Tangent space and dimension estimation}\label{chap5}

In this section, we combine the Propositions \ref{main empirical covariance estimation}, \ref{covariance lipschitz}, and \ref{main transport} to prove Theorem \ref{main}. This in turn implies both Theorem \ref{thmA} and \ref{thmB}.\footnote{Minor technical note: In the special cases discussed in the Introduction, we set $k=m$ in Theorems A and B, use Lemma \ref{simple log lemma}, and use $\log(14D) > 1 + \log (4D+2)$ assuming $D \ge 2$.}

\begin{defn}
    Given a $d$-dimensional subspace $\Pi \subseteq \RR^D$, denote the $D \times D$ orthogonal projection matrix to $\Pi$ by $\Proj_\Pi$, which is a real symmetric matrix, given concretely as:
     \[\Proj_\Pi = A_\Pi A_\Pi^\top \]
    where $A_\Pi \in \RR^{D \times d}$ is any matrix whose columns form an orthonormal basis of $\Pi$.
\end{defn}

\begin{defn}
	Let $\mathbf{X} = (X_1, \ldots X_m)$ be an i.i.d. sample drawn from $\mu$, a Borel probability measure on $\RR^D$. Given $x \in \RR^D$ and $r > 0$, define:
    \begin{align*}
    & \hat \Proj_{i} := \frac{d+2}{r^2} \Sigma[\delta_{\mathbf X_i}|_{U_i}] \text{, where } \mathbf X_i = \{X_j\}_{j \neq i}, U_i = \Ball_r(X_i) \\
    \end{align*}
\end{defn}
If $\Pi \subseteq \RR^D$ is a $d$-dimensional subspace, then Lemma \ref{covariance of disk} says that:
    \[ (d+2) \Sigma[\on{Unif}(\Pi \cap \Ball_1(0))] = P_\Pi \]
Thus an approximation to this covariance matrix in Proposition \ref{main transport} amounts to the approximation of a projection matrix, and justifies the definition of $\hat\Proj_{i}$.

\begin{thm}\label{main}
	Let $(\mu, \mu_0) \in \OurMeasures(M, s)$\footnote{See Definition \ref{our measures defn}.} where $M$ is a smoothly embedded compact $d$-dimensional manifold $M \subseteq \RR^D$ with reach $\tau$ and $s \ge 0$ is a real number. Let $\varphi$ be the probability density function of $\mu_0$ which satisfies $\| \varphi(x)-\varphi(y) \| \le \alpha \cdot \d_M(x,y)$. Let $X_1, \ldots X_m$ be an i.i.d. sample drawn from $\mu$ and let $X_1^\perp, \ldots X_m^\perp$ be their orthogonal projections to $M$. Given $\delta, \epsilon, \alpha > 0$ and assuming\footnote{Nothing is lost from this assumption since operator norm of the difference of two projection operators is at most 2.} $\epsilon < 2$, suppose $r, m$ satisfy the following: 
    \begin{align*}
        & \sqrt{\frac{2s}{\tau}} \le \frac{r}\tau \le \frac{\epsilon}{16(d+2)Q_2} \text{ and } \frac{m}{\log m} \ge \frac{4642(d+2)^2}{u_0 \epsilon^2} \log\left(\frac{14D\Proportion}{\delta}\right)
    \end{align*}
    where $u_0 = \inf_{x \in \on{supp} \mu} \mu(\Ball_r(x))$. Then with probability at least $1-\delta$, the following holds:
	\begin{align*}
	    & \max_{i \le \alpha m} \left\| \hat \Proj_i - \Proj_i \right\| \le \epsilon
	\end{align*}
	where $\Proj_i$ is the projection matrix to the tangent space $T_{X_i^\perp}M$, and $Q_2$ is defined as:
	\begin{align*}
	    & Q_2 = \QtwoDefinition \text{, where } \Phi = \frac{\mu_0(\Pi^{-1} \Ball_-^\circ )}{\omega_d r_-^d}
	\end{align*}
\end{thm}
\begin{proof}
    Out of total allowed error $\epsilon$, we will allocate one half $\epsilon/2$ to the concentration inequality (Proposition \ref{main empirical covariance estimation}) and the other half $\epsilon/2$ to the curvature (Proposition \ref{main transport}). Throughout the proof, we use the shorthand $U_i = \Ball_r(X_i^\perp)$.
    
    \textbf{Concentration inequality:} By Proposition \ref{main empirical covariance estimation}, we may use $k = \lfloor \Proportion m \rfloor$ points for local covariance estimation by error level $r^2\epsilon/2(d+2)$:
    \[ \left\| \Sigma[ \delta_{\mathbf X_i}|_{U_i}] - \Sigma[ \MeasRes{\mu}{U_i} ] \right\| \le \frac{r^2}{d+2} \cdot \frac{\epsilon}2 \text{, for all } i \le k \]
    with probability at least $1-\delta$, if $m$ satisfies the inequality in the theorem statement.
    
    \textbf{Curvature:} By combining Corollary \ref{main transport simplified} and Proposition \ref{covariance lipschitz}, the following holds\footnote{Applying Corollay \ref{main transport simplified} requires assuming $\rho + \rho^2 \le \sqrt2 - 1$ and $\rho \le 1/(\sqrt{8}d)$. But this assumption is automatically satisfied by the $r$ in the assumption of the theorem, where we already assume $\rho \le \epsilon / (16(d+2)Q_2) \le 1/(192 (d+2)^2)$. Thus these assumptions become redundant.} for every $x \in \on{supp} \mu$:
    \begin{align*}
    \left\| \Sigma[\MeasRes{\mu}{U_i}] - \frac{r^2}{d+2} \Proj_{i} \right\| \le 8r \cdot \frac{r^2 Q_2}\tau \le \frac{8 \tau \epsilon}{16(d+2)Q_2} \cdot \frac{r^2 Q_2}{\tau} = \frac{r^2}{d+2} \cdot \frac {\epsilon}2
    \end{align*}
    Note that $\frac{r^2}{d+2} \Proj_{X_i^\perp}$ is the covariance of the uniform measure over the tangential disk of radius $r$, by Lemma \ref{covariance of disk}.
    
    By the triangle inequality, for all $i \le k$ we have
    \begin{align*}
    \left\| \frac{d+2}{r^2} \Sigma[\delta_{\mathbf X_i}|_{U_i}] - \Proj_{i} \right\| & \le \frac{d+2}{r^2} \left( \left\| \Sigma[ \delta_{\mathbf X}|_{U_i}] - \Sigma[ \MeasRes{\mu}{U_i} ] \right\| + \left\| \Sigma[\MeasRes{\mu}{U_i}] - \frac{r^2}{d+2} P_{i} \right\| \right) \\
    & \le \frac{\epsilon}2 + \frac{\epsilon}2 = \epsilon,
    \end{align*}
    as desired. We note that the assumptions $2s \le r$ and $r+2s \le (\sqrt{2}-1)\tau$ of Proposition \ref{main transport} follow from the assumption on $r$ and $\epsilon < 2$.
\end{proof}

\subsection{Proof of Theorem A}

To use Theorem \ref{main}, we relate the projection matrices to angular deviation between subspaces using the Davis-Kahan theorem (see \cite{luxburg_davis_kahan}, \cite{davis_kahan}, \cite{davis_kahan_variant}).

To work with local behaviour of the space given by a union of two manifolds, we must understand the space of pair of subspaces. Indeed, every pair of linear subspaces of the same dimension can be characterised by \textit{principal angles}, up to (simultaneous) rigid motion. 

\begin{defn}
    Given $\pi_1, \pi_2 \in Gr(d, D)$, let $A_i \in \RR^{D \times d}$ be a matrix with orthonormal columns that span $\pi_i$. Suppose $\cos \theta_1 \ge \cdots \ge \cos \theta_d$ are singular values of the matrix $A_1^\top A_2$. The \textit{principal angles} of $(\pi_1, \pi_2)$ are defined as the angles $\pran(\pi_1, \pi_2) := (\theta_1, \ldots \theta_d)$, which satisfy $0 \le \theta_1 \le \cdots \le \theta_d \le \pi/2$.
\end{defn}

By abuse of notation, we will also refer to the largest principal angle $\theta_d$ as the "principal angle". This quantity has a simple interpretation:
\begin{lem}
    If $\pran(\pi_1, \pi_2) = (\theta_1, \ldots \theta_d)$ for $\pi_1, \pi_2 \in \on{Gr}(d, D)$, then:
    \[ \theta_d = \max_{x \in \pi_1} \min_{y \in \pi_2} \angle(x,y) \]
    Here $\angle(x,y) = \cos^{-1}(\langle x,y\rangle / (\| x \| \cdot  \| y \| ))$.
\end{lem}
\begin{proof}
    Let $A_i \in \RR^{D \times d}$ be a matrix whose columns form an orthonormal basis of $\pi_i$. We have:
    \begin{align*}
        \cos \theta_D = \min_{\|z\|=1} \|A_1^\top A_2 z\| = \min_{\|y\|=1, y \in \Pi_2} \|A_1^\top y\| = \min_{\|y\|=1, y \in \Pi_2} \langle y_1,  y \rangle
    \end{align*}
    where $y_1$ is the unit vector in the direction of $A_1 A_1^\top y$. Noting that $\langle y_1,y\rangle = \max_{\|x\|=1, x \in \pi_1} \langle x,y \rangle$, we have $\cos \theta_D = \min_{\|y\|=1, y\in\pi_2} \max_{\|x\|=1, x\in\pi_1} \langle x,y\rangle$.
\end{proof}

\begin{thm}[Davis-Kahan-Wang-Samworth]
	Let $A, B \in \RR^{D \times D}$ be real symmetric matrices. Let $1 \le d_1 \le d_2 \le D$ and assume that $\min ( \specgap_{d_1 - 1}A, \specgap_{d_2} A) > 0$, where $\specgap_k A = \lambda_k A - \lambda_{k+1} A$ is the $k$-th spectral gap of the matrix $A$. Let $\pi_A$ be the span of the eigenspaces of eigenvalues $\lambda_{d_1} A, \lambda_{d_1 + 1} A, \ldots \lambda_{d_2} A$, and let $\theta_1 \le \ldots \le \theta_d$ be the principal angles between $(\pi_A, \pi_B)$. Then we have:
	\[ \sqrt{\sin^2 \theta_{d_1} + \cdots + \sin^2 \theta_{d_2}} \le \frac2{\min(\specgap_{d_1 - 1} A, \specgap_{d_2} A)} \cdot \min\bigg( \|A-B\|_{\on F}, \> \sqrt{d} \|A-B\| \bigg) \]
	In particular, for $(d_1, d_2) = (1, d)$, we have:
	\[ \sqrt{\sin^2 \theta_{1} + \cdots + \sin^2 \theta_{d}} \le \frac{2 }{\specgap_d A} \cdot \min\bigg( \|A-B\|_{\on F}, \> \sqrt{d} \|A-B\| \bigg) \]
\end{thm}

We will only be using the case of $(d_1, d_2) = (1, d)$ above.

\vspace{5mm}
\textit{Proof of Theorem A.} This is a direct corollary of plugging in $\epsilon=(\sin\theta)/(2\sqrt{d+2})$ in Theorem \ref{main}. Assuming that, the following holds for each $i \le \lfloor \Proportion m \rfloor$:
\begin{align*}
    \| \Proj_{i} - \hat{\Proj}_i \| \le \frac{\sin\theta}{2\sqrt{d+2}}
\end{align*}
Since both $\Proj_{i}$ and $\hat\Proj_i$ are real symmetric matrices and since eigenvalues of $\Proj_{i}$ are $(1, \ldots 1, 0, \ldots 0)$, its $d$-th spectral gap is $1$ and therefore letting $A = \Proj_{i}$, $B = \hat\Proj_i$ in the Davis-Kahan theorem gives the following\footnote{In the equation, note that we could choose $\epsilon = \epsilon/(2\sqrt{d})$ for a slightly tighter bound. Our choice of $\epsilon$ is for cleanliness of the final expression produced.}:
    \[ \sin \measuredangle\bigg(\Pi(\Proj_{i}, d), \Pi(\hat \Proj_{i}, d)\bigg) \le 2\sqrt{d} \| \Proj_{i} - \hat \Proj_{i} \| \le \frac{2\sqrt{d}}{2\sqrt{d+2}} \sin\theta \le \sin \theta \]
Since $\Proj_{i}$ is the projection matrix to $T_{X_i^\perp}  M$, a $d$-dimensional space, we have $\Pi(\Proj_{i}, d) = T_{X_i^\perp}M$. Furthermore, $\Pi(\hat \Proj_i, d) = \Pi(\Sigma[\delta_{\Set{X}_i}|_{U_i}], d) = \hat\Pi_i$, where $U_i = \Ball_r(X_i)$.

In Theorem A, the conditions for $(r, m)$ used in Theorem \ref{main} are made stricter for the sake of easy interpretability. We explain how this is done.

\textbf{Condition on $r$.} The following is the required upper bound on $\rho = r/\tau$:
\begin{align*}
    \rho \le \frac{\epsilon}{16(d+2)Q_2} \text{, where } \epsilon = \frac{\sin\theta}{2\sqrt{d+2}}
\end{align*}
Using $\Phi \ge \varphi_{\min}$ (follows from Equation \eqref{eq: mu0 formula} and the Jacobian of inverse-projection being $\ge 1$), we get the following upper bound on $Q_2$:
\begin{align}\label{eq: Qtwo_upper}
    Q_2 =& \bigg( \frac72 + 8\sqrt{2} d \bigg) + \frac1{\Phi} \bigg( \frac{27d}2 \varphi_{\max} + 6\alpha \tau \bigg) \nonumber\\
    \le & \bigg( \frac72 + 8\sqrt{2} d \bigg) + \frac1{\varphi_{\min}} \bigg( \frac{27d}2 \varphi_{\max} + 6\alpha \tau \bigg) \nonumber\\
    \le & \bigg( \frac72 + 8\sqrt{2} + \frac{27}2 \bigg) d \cdot \frac{\varphi_{\max}}{\varphi_{\min}} + \frac{6\alpha\tau}{\varphi_{\min}} \nonumber\\
    \le & \frac{29 d \varphi_{\max} + 6\alpha\tau}{\varphi_{\min}}
\end{align}
Thus we get the required upper bound for $\rho=r/\tau$ used in Theorem A, as follows:
\begin{align*}
    \frac{\epsilon}{16(d+2)Q_2} = \frac{\sin\theta}{32(d+2)^{3/2}} \cdot \frac{\varphi_{\min}}{29d\varphi_{\max} + 6\alpha\tau} = \frac{\sin\theta}{(d+2)^{3/2}} \frac{\varphi_{\min}}{c_1 d\varphi_{\max} + c_2 \alpha\tau}
\end{align*}
where $(c_1, c_2) = (928, 192)$. 

\textbf{Condition on $m$.} The required lower bound for $m/\log m$ is obtained by also plugging in $\epsilon = \sin\theta/(2\sqrt{d+2})$ in Theorem \ref{main}, and noting that $u_0 \ge \omega_d r_-^d \varphi_{\min}$, by Equations \eqref{eq: mu0 formula} and \eqref{nice inclusions}. Furthermore, we use the following:
\begin{align*}
    r_- = r\bigg(1-\frac{r^2}{4\tau^2}\bigg) - 2s \ge r \cdot \bigg( 1 - \frac r{\tau} - \frac{r^2}{4\tau^2} \bigg)
\end{align*}
Assuming that $\rho=r/\tau$ satisfies $\rho + \rho^2/4 \le c/d$ for some constant $c>0$ and applying Lemma \ref{calculus lemma 3} with $t = \rho+\rho^2/4 \le c/d$, we get:
\[ \frac1{r_-^d} \le \frac1{r^d} \bigg( 1 - t \bigg)^{-d} \le \frac1{r^d} \bigg( 1 + \frac{d}{(1-c)^2} t \bigg) \le \frac1{r^d} \bigg( 1 + \frac{c}{(1-c)^2} \bigg) \]
By assuming the condition on $r$ derived above, we have that $\rho \le 1/(3^{3/2} \cdot 928) \le 4820$, so that we can take $c = 0.00025$, which implies $c / (1-c)^2 \le 0.0003$. This yields $1.0003 \times (4642 \times 4) \le 18574 = c_3$.

\subsection{Proof of Theorem B}

To relate a perturbation of eigenvalues to a perturbation of covariance matrices, we use the Hoffman-Wielandt theorem \cite{hoffman_wielandt}. 
\begin{thm}[Hoffman-Wielandt]\label{hoffman_wielandt}
	For normal matrices $A, A'$ of dimension $D \times D$, there is an enumeration of eigenvalues $(\lambda_1, \ldots \lambda_D)$ of $A$ and $(\lambda_1', \ldots \lambda_D')$ of $A'$ such that
	\[ \sum_{i=1}^D |\lambda_i - \lambda_i'|^2 \le \|A - A'\|_{\on{F}}^2\]
	where $\|A\|_{\on{F}} := \sqrt{\text{Tr}(A^\top A})$ denotes the Frobenius norm, with $\text{Tr}(\bullet)$ denoting the trace. In particular, if $A, A'$ are real symmetric matrices, then\footnote{This special case for real symmetric matrices follows from Lemma \ref{permutation matching distance}.}
	\[ \|\vec\lambda[A] - \vec\lambda[A']\| \le \|A-A'\|_{\on{F}} \]
	where $\vec\lambda[A] \in \RR^D$ is the vector of eigenvalues of $A$, arranged in the decreasing order.
\end{thm}

Now we note the following simple result for dimension estimation using tail sum.

\begin{prop}\label{tail prop}
    Let $\vec\lambda = (\lambda_1, \ldots \lambda_D) \in \RR^D$ be such that $\lambda_1 \ge \lambda_2 \ge \cdots \ge \lambda_D \ge 0$. Let $\vec\lambda(d,D) = \frac1{d+2}(1, \ldots 1, 0 \ldots 0) \in \RR^D$ where there are $D-d$ zeros. Let $\eta$ be a tolerance parameter such that $0 < \eta < 1/(2d)$.
    \[ \bigg\| \vec\lambda - \vec\lambda(d, D)\bigg\|_2 < \frac{1}{3\sqrt{D} (1+\eta^{-1})} \implies \on{Thr}(\vec\lambda, \eta) = d \]
    where $\on{Thr}$ is defined in the Introduction.
\end{prop}
\begin{proof}
    Writing $\vec\lambda - \vec\lambda(d, D) = (t_1, \ldots t_D)$, let $q_1 = |t_1| + \cdots + |t_d|$, $q_2 = |t_{d+1}| + \cdots + |t_D|$, and $q = q_1 + q_2 = \|\vec \lambda - \vec\lambda(d, D) \|_1$. Then since generally $D^{-1/2} \|x\|_1 \le \|x\|_2$, we have:
    \[ q < \sqrt{D} \cdot \frac{\eta}{3\sqrt{D} (1+\eta)} = \frac{\eta}{3(1+\eta)} \]
    A sufficient condition for $\on{Thr}(\vec\lambda, \eta) = d$ is:
    \begin{align*}
        & q_2 \le \eta \|\vec\lambda\|_1, \text{ and } q_2 + \left(\frac1{d+2} - q_1\right) > \eta \|\vec\lambda\|_1
    \end{align*}
    Since $\|\vec\lambda(d,D)\|_1 = d/(d+2)$, triangle inequality implies that $\frac d{d+2} - q \le \|\vec\lambda\|_1 \le \frac d{d+2} + q$. Thus we can formulate the following sufficient conditions:
    \begin{align*}
        & q < \eta \left( \frac d{d+2} - q \right), \text{ and } \frac1{d+2} - q > \eta \left( \frac d{d+2} + q \right) \\
        \iff & (1+\eta) q < \frac {\eta d}{d+2}, \text{ and } (1 + \eta)q < \frac{1 - \eta d}{d+2}\\
        \iff & q < \frac{\min(\eta d , 1 - \eta d)}{(1+\eta)(d+2)}
    \end{align*}
    By our assumption that $\eta < 1/(2d)$, we have $\min(\eta d, 1-\eta d) = \eta d$. Thus our sufficient condition is $q < \frac{\eta}{1+\eta} \cdot \frac d{d+2}$. The right hand side is minimised for $d=1$, so that this is precisely implied by the assumption.
\end{proof}

\vspace{5mm}
\textit{Proof of Theorem B.}

The proof goes verbatim except we use the Hoffman-Wielandt theorem instead of the Davis-Kahan theorem, and that we use the estimation error for the covariances $\|\hat \Sigma - \Sigma\|_2$, given by $\epsilon^{-1} = 3D(1+\eta^{-1})$. Then the following chain of inequalities hold with probability $\ge 1-\delta$:
\[ \|\vec\lambda - \vec\lambda(d, D)\|_2 \le \| \hat \Sigma - \Sigma \|_{\on{F}} \le \sqrt{D} \cdot \| \hat \Sigma - \Sigma \|_2 \le \frac1{3\sqrt{D}(1+\eta^{-1})} \]
The proof is then completed by applying Proposition \ref{tail prop}. We note how the expression $Q_2$ is weakened by using Equation \eqref{eq: Qtwo_upper}, which is also used in deriving Theorem A:
\begin{align*}
    \frac{\epsilon}{16(d+2)Q_2} \ge \frac1{48(d+2)D(1+\eta^{-1})} \frac{\varphi_{\min}}{29 d \varphi_{\max} + 6\alpha\tau} = \frac{1}{(d+2)D(1+\eta^{-1})} \frac{\varphi_{\min}}{c_1 d\varphi_{\max} + c_2 \alpha \tau}
\end{align*}
where $(c_1, c_2) = (1392, 288)$. The condition on $m$ is derived in a similar manner described in the proof of Theorem A. This time, we get $1.0003 \times (4642 \times 9) \le 41791 = c_3$.

\bibliographystyle{abbrv}
\bibliography{refs}

\pagebreak
\section{Appendix}
\subsection{Notations and conventions}
\label{notation appendix}

Here are some conventions we use.
\begin{itemize}
    \item The word `dimension' and `intrinsic dimension' are used interchangeably, where `intrinsic' simply distinguishes it from the `ambient' dimension.
    \item All manifolds are connected.
    \item All vectors are by default column vectors.
    \item $\|v\| = \sqrt{v^\top v}$ denotes the Euclidean norm of a vector $v \in \RR^D$.
    \item $\|A\|$ denotes the operator norm of a matrix $A \in \RR^{m \times n}$, seen as a map $\RR^n \rightarrow \RR^m$. $\|A\|_{\on{F}} = \sqrt{\on{Tr}(A^\top A)}$ denotes its Frobenius norm.
    \item $I_d$ denotes the $d \times d$ identity matrix.
    \item $\EE[X]$ denotes the expected value of a random variable $X$.
    \item $\Sigma[\mu]$ denotes the covariance matrix of a Borel probability measure $\mu$ over $\RR^D$.
    \item $\Ball_r(x) \subseteq \RR^D$ denotes the open ball of radius $r$ centered at $x \in \RR^D$.
    \item Given a smoothly embedded manifold $M \subseteq \RR^D$ and a point $x \in M$, $\ExpBall_r \subseteq T_x M$ denotes the open ball of radius $r$ centered at $0 \in T_x M$, assuming that the choice of $x$ is clear from the context.
    \item $\vec\lambda[A] \in \RR^D$ denotes the eigenvalues of a real symmetric matrix $A$ of size $D \times D$, arranged in the decreasing order.
    \item $\omega_d = \pi^{d/2} / \Gamma(\frac d2 + 1)$ denotes the volume of the $d$-dimensional unit ball.
\end{itemize}

Additionally, the following letters have specific meanings if not stated otherwise:
\begin{table}[H]
\begin{tabular}{|c||l|}
\hline
\textbf{Notation} & \textbf{Meaning} \\ \hline
$M$ & A compact manifold smoothly embedded in $\RR^D$ \\
$d$ & (Intrinsic) dimension of $M$ \\
$D$ & Ambient dimension \\
$\tau$ & Reach of $M$ \\
$\mu$ & Main distribution of interest with noise \\
$\mu_0$ & $\mu$ before adding noise \\
$\varphi$ & Probability density function on $M$ used to define $\mu_0$ \\
$\alpha$ & Lipschitz constant for $\varphi$ \\
$m$ & Sample size \\
$r$ & Local detection radius \\
$s$ & Noise radius \\
$\varrho$ & Probabilistic guarantees hold for $\lfloor \varrho m \rfloor$ out of $m$ points \\
$\delta$ & Probabilistic guarantees hold with probability $\ge 1-\delta$ \\
$\rho$ & Normalized local detection radius $\rho = r/\tau$ \\
$\sigma$ & Normalized noise radius $\sigma = s/\tau$ \\
\hline
\end{tabular}
\end{table}

\pagebreak
\subsection{Technical lemmas}

\begin{lem} \label{covariance of disk}
 (Lemma 13 from \cite{spectral_clustering_local_pca})
  Given a $d$-dimensional subspace $\Pi$ of $\RR^D$, the covariance matrix of the uniform distribution over the disk $\Pi \cap \Ball_1(0)$ is given by:
    \[ \Sigma[\Unif(\Pi \cap \Ball_1(0))] = \frac1{d+2} \Proj_\Pi \]
    where $\Proj_\Pi$ is the $D \times D$ projection matrix to $\Pi$. Eigenvalues of this matrix are:
    \begin{align*}
        \frac1{d+2} (\underbrace{1, \ldots 1}_d , \underbrace{0, \ldots 0}_{D-d} )
    \end{align*}	
\end{lem}
\begin{proof}
    Denote by $\Pi_{d,D}$ the $d$-dimensional subspace of $\RR^D$ spanned by the first $d$ canonical basis vectors. The only nontrivial covariance between the marginals of $\Unif(\Pi_{d, D} \cap \Ball_1(0))$ is:
	\[ \frac1{\omega_d} \int_{\|x\| \le 1} x_1^2 \d x = \frac1{\omega_d \cdot d} \int_{\|x\| \le 1} \|x\|^2 
	\d x = \frac1d \int_0^1 r^2 \cdot d r^{d-1} \d r = \int_0^1 r^{d+1} \d r = \frac1{d+2} \]
	where $1/{\omega_d}$ is multiplied so that the distribution is uniform over the unit disk. This yields the calculation for the vector of eigenvalues. Thus,
	\[ \Sigma[\Unif(\Pi_{d, D} \cap \Ball_1(0)] = \frac1{d+2} \begin{bmatrix} I_d & 0 \\ 0 & \mathbf{0}_{D-d} \end{bmatrix} \]
	
	Given any $d$-dimensional subspace $\Pi \subseteq \RR^D$, consider an orthonormal basis $A = [v_1, \ldots v_D]$ such that the first $d$ vectors $[v_1, \ldots v_d]$ span $\Pi$. If $X \sim \on{Unif}(\Pi \cap \Ball_1(0))$, then $A^{-1} X \sim \on{Unif}(\Pi_{d,D} \cap \Ball_1(0))$. Thus the covariance matrix of $X$ is
    \[ \frac1{d+2} A \begin{bmatrix} I_d & 0 \\ 0 & \mathbf{0}_{D-d} \end{bmatrix} A^\top = \frac1{d+2} [v_1, \ldots v_d] [v_1, \ldots v_d]^\top = \frac1{d+2} \Proj_{\Pi} \]
\end{proof}

\begin{lem}\label{unif eig distances}
Suppose
\[ \vec\lambda(d, D) := \frac1{d+2} (\underbrace{1, \ldots 1}_d , \underbrace{0, \ldots 0}_{D-d} ) \]
	If $d \le d'$, then
\begin{align*}
	& \| \vec \lambda(d, D) - \vec \lambda(d', D) \|^2 = \frac{(d'-d)(dd' + 4d + 4)}{(d+2)^2(d'+2)^2}
\end{align*}
	Also for any $k \neq d$, we have:
	\begin{align*}
	\| \vec\lambda(k, D) - \vec\lambda(d, D) \| \ge \| \vec \lambda(d, D) - \vec \lambda(d+1, D) \| =  \frac{\sqrt{(d+1)(d+4)}}{(d+2)(d+3)}
	\end{align*}
\end{lem}
\begin{proof}
	The norm of difference is given by direct computation:
	\[ \|\vec \lambda (d, D) - \vec \lambda (d', D) \|^2 = d \cdot \left( \frac{1}{d+2} - \frac{1}{d'+2} \right)^2 + \frac{d'-d}{(d'+2)^2} = \frac{(d'-d)(dd' + 4d + 4)}{(d+2)^2(d'+2)^2} \]
	
The partial derivative of the above expression with respect to $d$ and $d'$ are strictly negative and positive respectively, whenever $0 < d < d'$. Thus for each $d \ge 2$,
\begin{align*}
	& \min_{d' \neq d} \| \vec \lambda(d, D) - \vec \lambda(d', D) \| \\
	 =& \min( \| \vec \lambda(d, D) - \vec \lambda(d+1, D) \|, \| \vec \lambda(d, D) - \vec \lambda(d-1, D) \|) \\
	 =& \min\left( \frac{\sqrt{(d+1)(d+4)}}{(d+2)(d+3)}, \frac{\sqrt{d(d+3)}}{(d+1)(d+2)} \right) \\
	 =& \frac{\sqrt{(d+1)(d+4)}}{(d+2)(d+3)}
\end{align*}
where we use the fact that $\frac{\sqrt{(d+1)(d+4)}}{(d+2)(d+3)}$ is decreasing in $d$ for $d \ge 0$ (directly checked by computing the derivative of its square).
\end{proof}

Let's prove simple inequalities associated to optimal transport, constituting the main tools to obtain the necessary bounds for covariance matrices.

\begin{lem}\label{wass1}
Let $M$ be a Polish metric space with metric $\d_M$. Suppose $A, B \subseteq M$ are Borel measurable, with inclusion maps $\iota^A:A \hookrightarrow M, \iota^B: B \hookrightarrow M$. Suppose that there is a continuous bijection $f : A \rightarrow B$ with a $L \ge 0$ with $d_M(x, f(x)) < L$ for any $x$. Let $\mu$ be a Borel probability measure on $A$. Then for any $p \ge 1$, the Wasserstein distance between pushforwards of $\mu$ and $f_*\mu$ along inclusions are bounded by $L$:
\[ \Wass_p(\iota_*^A \mu, \iota_*^B f_* \mu) \le L\]
\end{lem}
\begin{proof}
If $X \sim \iota_*^A \mu$, then $f(X) \sim \iota_*^B f_* \mu$. Therefore, by using the coupling $(X, f(X))$, we obtain the claim:
\[ W_p(\iota_*^A \mu , \iota_*^B f_* \mu) \le (\EE_X \d_M (X, f(X))^p )^{1/p} \le L \]
\end{proof}

\begin{lem}\label{wass1_5}
    Let $M$ be a Polish metric space with metric $\d_M$ and a finite diameter $L := \sup_{x,y \in M} \d_M(x,y)$. For a Borel probability measure $\mu$ on $M$ and a Dirac delta measure $\delta_x$ centered at $x \in M$, we have:
    \[ \Wass_p(\mu, \delta_x) \le L\]
\end{lem}
\begin{proof}
    Define the transportation plan $\nu$ on $M \times M$ by 
    \[ \nu(U \times V) = \begin{cases} \mu(U) & \text{ if $x \in V$} \\ 0 & \text{ otherwise} \end{cases} \]
    whose marginals are $\mu$ and $\delta_x$. The transportation cost is bounded by $L$.
\end{proof}

\begin{lem}\label{wass2}
Let $M$ be a Polish metric space with metric $\d_M$ and a finite diameter $L := \sup_{x,y \in M} \d_M(x,y)$. Fix a Borel probability measure $\mu$ on $M$. Let $f$ be a non-negative continuous function on $M$ with $\sup_{x \in M} f(x) - \inf_{x \in M} f(x) \le C$ and $\int_M f(x) \text{d} \mu(x) = 1$. Let $\mu_f$ be the Borel probability measure on $M$ given by taking $f$ as the probability density function. Then for any $p \ge 1$,
\[ \Wass_p(\mu_f, \mu) \le C L \]
\end{lem}
\begin{proof}
    For any real number $a$, we have $a = \max(0, a) - \max(0, -a)$. Applying this to $a = f(x)-1$, we may write:
    \begin{align*}
        & \mu_f = \mu + \mu_f^+ - \mu_f^- \\
        \text{where } & \mu_f^+(U) = \int_U \max(0, f(x)-1) \d \mu(x) \\
        & \mu_f^-(U) = \int_U \max(0, 1-f(x)) \d \mu(x)
    \end{align*}
    As such, for any point $x \in M$,
    \[ \Wass_p(\mu_f, \mu) = \Wass_p(\mu + \mu_f^+ - \mu_f^-, \mu) \le \Wass_p(\mu_f^+, \mu_f^-)\]
    The inequality holds since generally $\Wass_p(\mu + \nu_1, \mu + \nu_2) \le \Wass_p(\nu_1, \nu_2)$. Since $\mu(M) = \mu_f(M)$, we have $A := \mu_f^+(M) = \mu_f^-(M)$. Then
    \[ \Wass_p(\mu_f^+, \mu_f^-) \le \Wass_p(\mu_f^+, A \cdot \delta_x) + \Wass_p(A \cdot \delta_x, \mu_f^-) \le 2AL \]
    The second inequality is by the previous lemma. By definition of $\mu_f^+, \mu_f^-$,
    \begin{align*}
        & A = \mu_f^+(M) \le \sup_{x \in M} f(x) - 1 \\
        & A = \mu_f^-(M) \le 1 - \inf_{x \in M} f(x)
    \end{align*}
    Thus $2A \le C$, and $2AL \le CL$.
\end{proof}

\begin{lem}\label{simple log lemma}
Suppose $a,b,x$ are real where $b > 1$ and $x>e$. Then we have that
\[ \frac{x}{\log x} > a(1+\log b) \implies x > a \log bx \implies \frac{x}{\log x} > a \]
\end{lem}
\begin{proof}
Writing $y = \log x > 1$ and $c = \log b>0$, the assertion follows trivially:
\[ x/y > a(1+c) \implies x > a(y+c) \implies x/y > a \]
\end{proof}

\begin{lem}\label{calculus lemma}
    For the following function
    \begin{align*}
        f(x) = \frac{1-ax}{(1+ax)(1 + x + ax^2)}
    \end{align*}
    the following holds whenever $a>0 , k\ge1$ and $x \in [0, 1/a]$:
    \begin{align*}
        f(x)^k \ge 1 - k (1+2a) x
    \end{align*}
\end{lem}
\begin{proof}
    Let's always assume $x \in [0, 1/a]$ here. By direct evaluation, $f'(0) = -(1+2a)$ and thus the claim is equivalent to $f(x)^k \ge 1 + kf'(0) x$. Since $f(0)=1$, it's sufficient to show that $(f^k)'(x) \ge k f'(0)$ for any $x$. We have $f'<0$ since $f$ is decreasing, and we can also directly check that $0 \le f \le 1$. Thus $(f^k)' = kf^{k-1}f' \ge kf'$. Thus it suffices to show that $f'\ge f'(0)$. By direct computation, we have:
    \begin{align*}
        f'(x) = \frac{2a^3 x^3 - (a^2x^2 + 4ax + 2a + 1)}{(1+ax)^2(1+x+ax^2)^2}
    \end{align*}
    We want $f' \ge f(0) = -(1+2a)$, which is equivalent to:
    \begin{align*}
        2a^3 x^3 - (a^2x^2 + 4ax + 2a + 1) + (1+2a)(1+ax)^2(1+x+ax^2)^2 \ge 0
    \end{align*}
    which holds since all of the coefficients are positive, upon expanding the brackets.
\end{proof}

\begin{lem}\label{calculus lemma 2}
    For every $t>0$ and $s>1$, the following hold:
    \begin{align*}
        & \frac1{1-e^{-1/t}} -t \in [\frac12, 1] \\
        & \frac{1}{\log(1-s^{-1})} +s \in [\frac12, 1]
    \end{align*}
    Furthermore, both functions are increasing.
\end{lem}
\begin{proof}
    The function $s(t) = 1/(1-e^{-1/t})$ is an increasing bijection from $(0, \infty)$ to $(1, \infty)$ and we have $t = -1/\log(1-s(t)^{-1})$. Thus it suffices to prove the properties regarding the function:
    \[ f(t) = \frac1{1-e^{-1/t}}-t = \frac{e^{u}}{e^{u}-1} - \frac1u = \frac{ue^u - e^u + 1}{u(e^u - 1)} \text{, where } u = \frac1t \]
    Then the claim that this quantity falls in the interval $[1/2, 1]$ is equivalent to:
    \[ ue^u - u \le 2ue^u - 2e^u + 2 \text{, and }  ue^u - e^u + 1 \le ue^u - u \]
    or equivalently,
    \[ 0 \le (u-2)e^u + (u+2) \text{, and } 1+u \le e^u \]
    The second inequality is a standard fact, and plugging it into the first inequality shows it easily. To show that $f(t)$ is increasing, we evaluate the derivative:
    \[ \frac{\d}{\d t} \left(\frac1{1-e^{-1/t}} -t \right) = \frac{e^{1/t}}{(e^{1/t}-1)^2 t^2} - 1 \]
    The derivative is positive iff:
    \[ \frac1{t^2} \le  \frac{(e^{1/t}-1)^2}{e^{1/t}} \]
    which follows from the following:
    \[ u \le u \sum_{k=0}^\infty \frac{(u/2)^{2k}}{(2k+1)!} = e^{u/2} - e^{-u/2} \text{, where } u = \frac1{t} \]
\end{proof}

\begin{lem}\label{calculus lemma 3}
    Suppose $0 < c \le 1, d \ge 1$ and $t \le c/d$. Then we have the following linear bound:
    \[ (1-t)^{-d} \le 1 + \frac d{(1-c)^2} \cdot t \]
\end{lem}
\begin{proof}
    Let $f_d(t)$. The first and second derivatives are:
    \[ f_d'(t) = d(1-t)^{-d-1}, \> f_d''(t) = d(d+1)(1-t)^{-d-2}\]
    and thus $f_d'(t)$ is an increasing function at $t\in[0,1]$. This implies that, for each $0 \le t \le t_0 \le 1$, we have:
    \[ f_d(t) \le 1 + f_d'(t_0) t \]
    Take $t_0 = c/d$. Then:
    \[ f_d'(c/d) = \frac{d}{1-c/d} \cdot \frac1{(1-c/d)^d} \le \frac d{(1-c)^2} \]
    where we used the fact that $s \mapsto (1-1/s)^s$ is an increasing function for $s \ge 0$ to see that $(1-c/d)^d \ge (1-c)$.
\end{proof}

\begin{lem}\label{calculus lemma 4}
    Suppose $d \ge 1, t \in [0,1]$. Then
    \[ \bigg(\frac{1-t}{1+t}\bigg)^d \ge 1-2d\cdot t \]
\end{lem}
\begin{proof}
    The first and second derivative of the function $f_d(t) = ((1-t)/(1+t))^d$ are:
    \[ f_d'(t) = \frac{2d}{t^2-1} \bigg( \frac{1-t}{1+t}\bigg)^d, \> f_d''(t) = \frac{4d(d-t)}{(t^2-1)^2}\bigg(\frac{1-t}{1+t}\bigg)^d \]
    For $t \in [0,1]$, the second derivative is $\ge 0$. Therefore we have $f_d(t) \ge 1 + f_d'(0)t$. Since $f_d'(0) = -2d$, we get the claim.
\end{proof}

\begin{lem}
    Let $M \subset \RR^D$ be a compact set and let $\tau$ be its reach. Let $\pi_M$ be the projection map to $M$, such that for any $x \in \RR^D$, $\pi_M(x)$ is the set of points on $M$ that minimises the distance to $M$. The following hold:
    \begin{enumerate}
        \item The distance function $x \mapsto \d(x, M) = \inf \{\|y-x\| \>|\> y \in M\}$ is continuous.
        \item For $0 < r < \tau$, $\pi_M|_{\Ball(M, r)}$ is a single-valued continuous function.
    \end{enumerate}
\end{lem}
\begin{proof}
    (1) From the definition it easily follows that $\d(-, M)$ is a Lipschitz function; we have that: $|\d(x, M) - \d(x', M)| \le \|x-x'\|$.

    (2) Let's write $\pi = \pi_M|_{\Ball(M,r)}$ for the moment. Let $x \in \Ball(M, r)$. Suppose that $x_n \rightarrow x$ but $\pi(x_n)$ doesn't converge to $\pi(x)$. Then there exists $s>0$ such that $\pi(x_n) \notin \Ball(\pi(x), s)$.

    Since $\d(y, M) = \|y - \pi(y)\|$ for each $y \in \Ball(M, r)$, the continuity of $\d(-, M)$ implies that there is a convergence $\|x_n - \pi(x_n)\| \rightarrow \|x - \pi(x)\|$. Since we also have $x_n \rightarrow x$, we have $\|x - \pi(x_n)\| \rightarrow \|x - \pi(x)\|$. Thus $\inf \{ \|x - y\| \>|\> y \in M \backslash \Ball(\pi(x), s) \} = \|x - \pi(x)\| = \d(x, M)$.
    
    This is a contradiction. Since $M \backslash \Ball(\pi(x), s)$ is a compact set, the distance function $y \mapsto \|y-x\|$ attains a minimum on some $z \in M \backslash \Ball(\pi(x), s)$. This violates the definition of reach, which requires a unique nearest point of $x$ on $M$, which can't be simultaneously $\pi(x)$ and $z$.
\end{proof}

\begin{lem}\label{lem: locally connected}
    Let $M \subset \RR^D$ be a compact path-connected set and let $\tau$ be its reach. If $x, y \in M$ satisfies $\|x-y\| < \tau$, then there exists a continuous path on $M$ that connects $(x,y)$ such that every point on the path has distance at most $\|x-y\|$ from both $x$ and $y$.
\end{lem}
\begin{proof}

    Define a path $\bar\gamma:[0,1] \rightarrow M$ by $\bar \gamma(t) = (1-t)x + ty$, the line segment connecting $(x,y)$. Since $\|x-y\| < \tau$, every point on $\bar\gamma$ is within distance $\tau$ from $x$, and thus $\pi_M \circ \bar\gamma: [0,1] \rightarrow M$ is a (single-valued) continuous function. Let's write $\gamma = \pi_M \circ \bar\gamma$.

    Let $t_0 \in [0,1]$ and write $z = \gamma(t_0)$ and $\bar z = \bar \gamma(t_0)$. Then we have:
    \[ \|z - x \| \le \|z - \bar z \| + \|\bar z - x\| \le \|y - \bar z \| + \|\bar z - x\| = \|y - x\| \]
    where the first inequality is the triangle inequality, the second inequality is due to the definition of $\gamma$, and the last equality is due to $(x, \bar z, y)$ lying on one line. Therefore $\|z - x \| \le \|y-x\|$, and by symmetry of the argument in $(x,y)$, we also get $\|z-y\| \le \|y-x\|$.
\end{proof}

\begin{prop}\label{jacobian lemmas}
    Let $M$ be a $d$-dimensional submanifold. Let $\pi_x: \RR^D \rightarrow T_x M$ be the projection map to $T_x M$, and let $\tilde \pi_x := \pi_x|_M : M \rightarrow T_x M$ and $\tilde \pi_{x, r} := \pi_x|_{M \cap \Ball(x, r)}$. The following hold:
    \begin{enumerate}
        \item When $r < \tau / 2$, $\tilde \pi_{x, r}$ has nonsingular derivatives and is a diffeomorphism.
        \item For any $y \in M$, we have $J_y \tilde \pi_x = \det(A_x^\top A_y)$, where $A_x \in \RR^{D \times d}$ is any orthonormal frame of $T_x M$.
        \item For any $y \in M$, the following bound holds:
        \[ \cos \theta_{x, y} \ge 1 - \frac{\d_M(x, y)}{\tau} \]
        where $\theta_{x,y} = \measuredangle_{\max}(T_x M, T_y M)$.
        \item For any $y \in M$, the following bound holds:
        \[ J_y \tilde \pi_x \in \bigg[ (\cos \theta_{x, y} )^d, 1 \bigg] \]
        If $r < (\sqrt{2}-1)\tau$, then we furthermore get:
        \[ J_y \tilde \pi_x \in \bigg[ (1-\sqrt{2}r/\tau)^d, 1 \bigg] \]
    \end{enumerate}
\end{prop}
\begin{proof}
    (1) The nonsingularity is Lemma 5.4 from \cite{nsw}. By applying the inverse function theorem locally at each point where the derivative is non-singular, we see that $\tilde \pi_{x, r}$ is a diffeomorphism.

    (2) This is because $\d \tilde {\pi_x} (v) = A_x^\top v$ for each (embedded) tangent vector $v \in T_y M$. 

    (3) This is Proposition 6.2 from \cite{nsw}. 

    (4) The first bound folllows from (2) and the definition of principal angles. The second bound follows from (3) and Lemma \ref{geodesic bound in ball}, which implies $\d_M(x,y)/\tau \le (r/\tau) + (r/\tau)^2 \le \sqrt{2}r/\tau$. 
\end{proof}

\begin{lem}\label{radial scaling jacobian}
	Let $f_0: \RR^{d} \rightarrow \RR^+$ be a function such that $f_0(x) = f_0(\lambda x)$ for any $\lambda > 0$, and that $f_0$ is differentiable when restricted to the unit sphere $S^{d-1}$. Define the scaling map $f(x) = f_0(x) x$ for $x \neq 0$. Then the Jacobian determinant of $f$ is given by:
	\[\on{J} f (x) = f_0(x) \]
\end{lem}
\begin{proof}
	We have that $\frac{\partial}{\partial x_j} (f_0(x) x_i) = \delta_{ij} \varphi + \frac{\partial f_0}{\partial x_j} x_i$ where $\delta_{ij}$ is the Kronecker delta. Then
	\[\on{J} f = \det( f_0 I_d + (\nabla g) x^\top ) = f_0 + (\nabla f_0)^\top x = f_0\]
	by the matrix determinant lemma and the fact that the directional derivative of $f_0(x)$ along $x$ is zero.
\end{proof}

The following lemma, which is a simple extension of Proposition 6.3 of \cite{niyogi_smale_weinberger}, controls the deviation of geodesic from the first order approximation:
\begin{lem}\label{geodesic bound in ball}
	Let $M$ be a smooth compact $n$-manifold embedded in $\RR^D$ with reach $\tau$. Suppose that $x,y$ are connected by a (unit speed) geodesic $\gamma: [0, \tilde r] \rightarrow M$ of length $\tilde r$ with $\gamma(0) = x, \gamma(\tilde r) = y$, and denote $r = \|x-y\|$. Then the following inequalities hold:
	\[ \tilde r - \frac{\tilde r^2}{2\tau} \le r \le \tilde r \]
	If $r \le 0.5 \tau$, then the following hold:
	\begin{align*}	
	\frac{\tilde r}\tau \le 1 - \sqrt{1 - \frac{2r}{\tau} } \text{, and } \| y - (x + \tilde r \dot \gamma(0)) \| \le \frac{\tilde r^2}{2\tau}
	\end{align*}	
If $r \le (\sqrt{2} - 1)\tau \approx 0.4 \tau$, then the following also hold:
	\begin{align*}	
	\tilde r \le r + \frac{r^2}{\tau} \text{, and } \| y - (x + \tilde r \dot \gamma(0)) \| \le \frac{r^2}{\tau}
	\end{align*}	
\end{lem}
\begin{proof}
    Since straight lines are geodesics in $\RR^D$, we have $r \le \tilde r$. Meanwhile by the triangle inequality,
	\[ r = \|\gamma(\tilde r) - \gamma(0)\| \ge  \| \tilde r \dot \gamma(0) \| - \left \|  \int_0^{\tilde r} \int_0^{t_1} \ddot \gamma(t_2) \d t_2 \d t_1 \right \| \ge \tilde r - \frac{\tilde r^2}{2\tau}\] 
	When $r \le \tau/2$, this is equivalent to $\tilde r \notin (\tau - \tau \sqrt{1-2\tau^{-1} r } , \tau + \tau \sqrt{1-2\tau^{-1} r } )$. Since $\tilde r = 0$ when $r =0$, by continuity we must have $\tilde r \le \tau - \tau \sqrt{1-2\tau^{-1} r}$. 
		
	To get the error bound of first-order approximation, we calculate by basic calculus the following:
	\begin{align*}
		& \gamma(\tilde r) - \gamma(0) = \int_0^{\tilde r} \dot \gamma(t_1) \d t_1  = \int_0^{\tilde r} \left( \dot \gamma(0) + \int_0^{t_1} \ddot \gamma(t_2) \d t_2 \right)  \d t_1 = \tilde r \dot \gamma(0) + \int_0^{\tilde r} \int_0^{t_1} \ddot \gamma(t_2) \d t_2 \d t_1
	\end{align*}
	and thus
	\begin{align*}
		& \|\gamma(\tilde r) - (\gamma(0) + \tilde r \dot \gamma(0)) \| =  \left\| \int_0^{\tilde r} \int_0^{t_1} \ddot \gamma(t_2) \d t_2 \d t_1 \right\| \le \int_0^{\tilde r} \int_0^{t_1} \frac1\tau \d t_2 \d t_1 = \frac{{\tilde r}^2}{2 \tau}
	\end{align*}	
	where the last inequality holds because for any $t$, $\|\ddot \gamma(t)\| \le \tau^{-1}$ (the norm of the second fundamental form is bounded above by $\tau^{-1}$. See Proposition 6.1 of \cite{niyogi_smale_weinberger}).
	
	To get simpler bounds, now suppose that $r \le (\sqrt 2 - 1) \tau$. We note that $x \in [0, \sqrt{2} - 1]$ implies\footnote{Since $(x+x^2)/(1-\sqrt{1-2x}) \in [1, 1.07]$ when $x \in [0, \sqrt{2}-1]$, this relaxation overestimates by at most 7 percent.} $1-\sqrt{1-2x} \le x + x^2$. Thus 
	\begin{align*}
		& \tilde r \le \tau - \tau \sqrt{1-2\tau^{-1} r} \le r + \frac{ r^2}{\tau} \\
		& \|\gamma(\tilde r) - (\gamma(0) + \tilde r \dot \gamma(0)) \| \le \frac{\tilde r^2}{2\tau} \le \frac{r^2}{2\tau^3} (r + \tau)^2 \le \frac{r^2}{\tau}
	\end{align*}
\end{proof}

Sectional curvature may be used to bound the Jacobian of the exponential map, as follows\cite{mario_thesis}:
\begin{thm}
    Let $M$ be a Riemannian manifold with sectional curvature bounded below and above by $\kappa_-$ and $\kappa_+$. Then for $x \in M$ and $v \in T_x M$, the following holds:
    \[ \min \left( 1, \frac{\on{sin} \sqrt{\kappa_+ } \|v\|}{\sqrt{\kappa_+ }\|v\|} \right) \le \| (\d \exp_x)_v \| \le \max \left( 1, \frac{\on{sin} \sqrt{\kappa_- } \|v\|}{\sqrt{\kappa_- }\|v\|} \right)\]
    for all $\|v\|$ if $\kappa_+ \le 0$, and for $\|v\| \le \pi/\sqrt{\kappa_+}$ otherwise. The quantity $\frac{\sin x}{x}$ is taken to be 1 when $x=0$.
\end{thm}
This implies a weaker bound given in terms of the reach:
\begin{cor}\label{exp jacobian bound}
    Let $M \subseteq \RR^D$ be a smoothly embedded compact Riemannian manifold with reach $\tau$. Then for $x \in M$ and $v \in T_x M$ satisfying $r := \|v\| \le \pi \tau$, we have:
    \[ \frac{\sinh \sqrt{2} \tau^{-1} r}{\sqrt{2} \tau^{-1} r} \le \| (\d \exp_x)_v \| \le \frac{\sin \tau^{-1} r}{\tau^{-1} r} \]
    In particular, if $r \le 2 \tau$, then
    \[ 1 - \frac{r^2}{6\tau^2} \le \| (\d \exp_x)_v \| \le 1 + \frac{r^2}{2\tau^2} \]
\end{cor}
\begin{proof}
Norm of the second fundamental form is bounded above by $\tau^{-1}$ \cite{niyogi_smale_weinberger}, and thus by the Gauss equation applied to sectional curvature (i.e. $K(u,v) = \langle R(u,v)u, v\rangle = \langle \II(u,u), \II(v,v) \rangle - \|\II(u,v)\|^2$ for orthonormal $u,v$), we may take $\kappa_- = - 2 \tau^{-2}$ and $\kappa_+ = \tau^{-2}$ for the curvature bounds. Thus the radius condition reads $r \le \pi \tau$. Then we have:
\begin{align*}
    & \frac{\sin \sqrt{\kappa_+} r}{\sqrt{\kappa_+} r} = \frac{\sin \tau^{-1} r}{\tau^{-1} r} = 1 - \frac{r^2}{6 \tau^2} + O(r^4) \ge 1 - \frac{r^2}{6\tau^2} \\
    & \frac{\sin \sqrt{\kappa_-} r}{\sqrt{\kappa_-} r} = \frac{\sinh \sqrt{2} \tau^{-1} r}{\sqrt{2} \tau^{-1} r} = 1 + \frac{r^2}{3 \tau^2} + O(r^4) \le 1 + \frac{r^2}{2\tau^2} \text{ for $r \le 2 \tau$}
\end{align*}
where in the end we used $\sinh x \le x + \frac{x^3}4$ for $x \in [0, 2\sqrt{2}]$\footnote{This can be manually checked by computing the first and the second derivative of $x + x^3/4 - \sinh x$.}.
\end{proof}

\begin{lem}\label{permutation matching distance}
    For a metric space $M$ and its $n$-fold product space $M^n$, the following function is a metric on $M^n$:
    \[ \d_\circ(x,y) := \min_{\sigma, \tau \in S_n} \d_M(\sigma \cdot x, \tau \cdot y)  = \min_{\sigma \in S_n} \d_M ( x, \sigma \cdot y) \]
    where $S_n$ is the permutation group on $n$ elements and $\sigma \cdot (y_1, \ldots y_n) = (y_{\sigma(1)}, \ldots y_{\sigma(n)})$ permutes the coordinates. If $M = \RR$, $x, y \in M$, and if entries of $x, y$ are arranged in the decreasing order, then
	\[ \d_\circ(x,y) = \| x-y \| \]
\end{lem}
\begin{proof}
Reflexivity and symmetry of $d_\circ$ hold obviously. To see the triangle inequality, suppose that $x, y, z \in M^D$ and define $\sigma_{xy}$ by the relation $d_\circ(x, y) = \d_M(x, \sigma_{xy} \cdot y)$ (similarly for $\sigma_{yz}, \sigma_{xz}$). Then 
\begin{align*}
	\d_\circ(x,y) + \d_\circ(y,z) =& \d_M(x, \sigma_{xy} \cdot y ) + \d_M( y, \sigma_{yz} \cdot z ) \\
	=& \d_M(x, \sigma_{xy} \cdot y ) + \d_M( \sigma_{xy} \cdot y , \sigma_{xy} \cdot \sigma_{yz} \cdot z ) \\
	\ge& \d_M( x,  \sigma_{xy} \cdot \sigma_{yz} \cdot z ) \\
	\ge& \d_\circ(x,z)
\end{align*}
This shows that $\d_\circ$ is indeed a metric.

Consider $M = \RR$. Suppose that $x_1 \le \cdots \le x_n, y_1 \le \cdots \le y_n$. Then we claim that for any $\sigma \in S_n$, $\|x - y\| \le \|x - \sigma \cdot y \|$. Suppose $z \in \RR^n$ doesn't necessarily have its entries ordered in a decreasing order. If there exists a pair $i < j$ with $z_i > z_j$, then we have: $\|x- \tau_{ij} \cdot z \| < \|x - z\|$,
where $\tau_{ij}\in S_n$ is the transposition that swaps $i$ and $j$. This is because whenever $a < b, a' < b'$, we have $(a-a')^2 + (b-b')^2 < (a-b')^2 + (b-a')^2$. By repeatedly applying this sorting process to $z = \sigma \cdot y$, we get the claim. The sorting process ends in finite time because one can recursively take the smallest unsorted element and swap it all the way down, i.e. perform a bubble sort.
\end{proof}

\end{document}